\documentclass[11pt,reqno]{amsart}

\usepackage{amsmath, amssymb, amsthm, amsfonts}
\usepackage{mathtools}
\usepackage{thmtools}
\usepackage{mathrsfs}
\usepackage{mathabx}
\newcommand{\quotient}[2]{\left.#1\!\middle/\!\raisebox{-.5ex}{$#2$}\right.}

\usepackage[T1]{fontenc}
\usepackage[utf8]{inputenc}
\usepackage{enumitem}
\usepackage{array, multirow}
\usepackage{float}
\usepackage[
colorlinks=true,
linkcolor=blue!50!black,
citecolor=green!40!black,
urlcolor=blue!50!black
]{hyperref}
\usepackage{cleveref}

\usepackage[
backend=biber,
style=alphabetic,
backref=true,
doi=false,      % Suppresses DOI fields
isbn=false,     % Suppresses ISBN and ISSN fields
url=false,      % Suppresses URL fields
eprint=false,   % Suppresses ArXiv / eprint fields
giveninits=false,
maxbibnames=99
]{biblatex}

\DeclareNameAlias{author}{given-family}
\DeclareNameAlias{editor}{given-family}

\DeclareFieldFormat[article]{title}{\mkbibemph{#1}}
\DeclareFieldFormat[incollection]{title}{\mkbibemph{#1}}
\DeclareFieldFormat[book]{title}{#1}

\DeclareFieldFormat[article]{journaltitle}{#1}
\DeclareFieldFormat[article]{volume}{\mkbibbold{#1}}
\DeclareFieldFormat[article]{number}{#1}
\DeclareFieldFormat[article]{pages}{#1}

\DeclareSourcemap{%
\maps[datatype=bibtex]{%
\map{%
\step[fieldsource=mrnumber, fieldtarget=usera]
}
}
}

\DeclareFieldFormat{usera}{%
\iffieldundef{usera}
{}
{\href{https://mathscinet.ams.org/mathscinet/article?mr=#1}{MR\texttt{#1}}}%
}

\renewbibmacro*{journal+issuetitle}{%
\usebibmacro{journal}%
\setunit*{\addcomma\space}%
\printfield{volume}%
\setunit*{\addcolon}%
\printfield{number}%
\setunit{\addspace}%
\printtext[parens]{\printdate}%
}

\renewbibmacro*{note+pages}{%
\setunit{\addcomma\space}%
\printfield{pages}%}%
}

\usepackage{xpatch}

\xpretobibmacro{pageref}{%
\iffieldundef{usera}
{}
{\setunit{\addperiod\space}\printfield{usera}\setunit{\space}}%
}{}{}

\renewbibmacro{in:}{%
\ifentrytype{article}
{}
{\printtext{\bibstring{in}\intitlepunct}}%
}

\DefineBibliographyStrings{english}{%
byeditor = {Edited by},
}

\AtEveryBibitem{%
\ifentrytype{book}{\clearfield{pages}\clearfield{pagetotal}}{}%
\ifentrytype{collection}{\clearfield{pages}\clearfield{pagetotal}}{}%
\ifentrytype{inbook}{\clearfield{pages}\clearfield{pagetotal}}{}%
}

\DeclareFieldFormat{edition}{%
\ifinteger{#1}
{\mkbibordedition{#1}~ed\adddot}
{#1~ed\adddot}%
}

\DeclareFieldFormat[misc]{eprint:arxiv}{%
arXiv\addcolon\allowbreak\hspace{0pt}#1%
}
\DeclareFieldFormat[misc]{url}{\url{#1}}

\DeclareBibliographyDriver{misc}{%
\usebibmacro{bibindex}%
\usebibmacro{begentry}%
\printnames{author}%
\setunit{\addperiod\space}%
\printfield{title}%
\setunit{\addperiod\space}%
\printfield{note}
\setunit{\addperiod\space}%
\iffieldundef{eprint}
{%
\iffieldundef{url}
{}
{\printfield{url}}%
}
{%
\printfield[eprint:arxiv]{eprint}%
\setunit{\addspace}%
\printtext[parens]{\printdate}\addperiod%
}
\usebibmacro{pageref}%
\usebibmacro{finentry}%
}

\DeclareFieldFormat[thesis,phdthesis,mastersthesis]{title}{\mkbibitalic{#1}}

\usepackage{xcolor}
\usepackage[textsize=tiny]{todonotes}
\makeatletter
\renewcommand\section{%
\@startsection{section}{1}\z@
{.7\linespacing \@plus \linespacing}
{.5\linespacing}{\normalfont\bfseries\scshape\centering}%
}

\renewcommand\subsubsection{%
\@startsection{subsubsection}{3}\z@ {.5\linespacing \@plus .7\linespacing }{-.5em}{\normalfont\bfseries}
}

\usepackage{scalerel}

\newcommand*{\medcap}{\mathbin{\scalebox{1.25}{$\cap$}}}

\newcommand{\interior}[1]{{#1}^{\mathrm{o}}}

\newcommand\preceqdot{\mathrel{\mathcal{O}align{$\preceq$\cr
\hidewidth\raise0.125ex\hbox{$\cdot$}\cr}}}
\newcommand\precdott{\mathrel{\mathcal{O}align{$\prec$\cr
\hidewidth\raise0.015ex\hbox{$\cdot$}\cr}}}

\declaretheoremstyle[
headfont=\bfseries,
bodyfont=\itshape,
]{plainstyle}

\declaretheoremstyle[
headfont=\bfseries,
bodyfont=\normalfont,
]{defstyle}

\declaretheorem[
style=plainstyle,
numberwithin=section,
name=Theorem
]{thm}

\declaretheorem[sibling=thm, style=plainstyle, name=Proposition]{prop}
\declaretheorem[sibling=thm, style=plainstyle, name=Lemma]{lemma}

\declaretheorem[sibling=thm, style=plainstyle, name=Corollary]{corollary}

\declaretheorem[sibling=thm, style=defstyle, name=Remark]{remark}

\declaretheorem[name=Claim,style=plainstyle]{claim}

\declaretheorem[numbered=no, name=Claim]{claim*}
\declaretheorem[numbered=no, name=Problem]{problem*}



\crefname{thm}{theorem}{theorems}
\Crefname{thm}{Theorem}{Theorems}

\crefname{prop}{proposition}{propositions}
\Crefname{prop}{Proposition}{Propositions}

\crefname{lemma}{lemma}{lemmas}
\Crefname{lemma}{Lemma}{Lemmas}

\crefname{fact}{fact}{facts}
\Crefname{fact}{Fact}{Facts}

\crefname{corollary}{corollary}{corollaries}
\Crefname{corollary}{Corollary}{Corollaries}

\crefname{conjecture}{conjecture}{conjectures}
\Crefname{conjecture}{Conjecture}{Conjectures}

\crefname{definition}{definition}{definitions}
\Crefname{definition}{Definition}{Definitions}

\crefname{remark}{remark}{remarks}
\Crefname{remark}{Remark}{Remarks}

\crefname{example}{example}{examples}
\Crefname{example}{Example}{Examples}

\crefname{exercise}{exercise}{exercises}
\Crefname{exercise}{Exercise}{Exercises}

\crefname{question}{question}{questions}
\Crefname{question}{Question}{Questions}

\crefname{claim}{claim}{claims}
\Crefname{claim}{Claim}{Claims}

\crefname{case}{case}{cases}
\Crefname{case}{Case}{Cases}

\crefname{problem}{problem}{problems}
\Crefname{problem}{Problem}{Problems}

\newcommand{\zero}{\boldsymbol{0}}

\newcommand{\vx}{\boldsymbol{x}}
\newcommand{\vy}{\boldsymbol{y}}

\newcommand{\vs}{\boldsymbol{s}}

\newcommand{\vv}{\boldsymbol{v}}

\newcommand{\vu}{\boldsymbol{u}}
\newcommand{\vz}{\boldsymbol{z}}
\newcommand{\vw}{\boldsymbol{w}}
\newcommand{\bbP}{\mathbb{P}}

\def\R{\mathbb{R}}
\def\Z{\mathbb{Z}}
\def\N{\mathbb{N}}

\def\fp{\mathfrak{p}}

\def\fd{\mathfrak{d}}
\def\bbQ{\mathbb{Q}}
\def\vol{\mathrm{vol}}
\def\dens{\mathrm{dens}}
\def\vis{\mathrm{vis}}
\def\pr{\mathrm{pr}}
\def\internal{\mathrm{int}}
\def\phys{\mathrm{phys}}

\def\EE{\mathcal{E}}

\def\SL{\mathrm{SL}}
\def\GL{\mathrm{GL}}
\def\ASL{\mathrm{ASL}}
\def\Cl{\mathscr{C}}
\def\scrX{\mathscr{L}}
\def\scrY{\mathscr{A}}

\def\haar{\mathrm{m}}

\def\df{{\, \stackrel{\mathrm{def}}{=}\, }}
\def\covol{\mathrm{covol}}
\def\cW{\mathcal{W}}
\def\cL{\mathcal{L}}
\def\cO{\mathcal{O}}
\def\cP{\mathcal{P}}
\renewcommand{\emptyset}{\varnothing}
\newcommand{\card}[1]{\operatorname{\#}\!\left(#1\right)}
\title {Simultaneous visibility in cut-and-project sets}

\author[R. Kumar]{Rishi Kumar}
\address{School of Mathematical Sciences\\
Tel Aviv University\\
Tel Aviv\\
69978\\
Israel}
\email{rkumar@tauex.tau.ac.il}

\author[C. Ospina]{Carlos Ospina}
\address{School of Mathematical Sciences\\
Tel Aviv University\\
Tel Aviv\\
69978\\
Israel}
\email{ospina.math@icloud.com}

\begin{document}

\begin{abstract} 
Visibility from the origin in integer lattices is a  well-studied problem. More recently, the corresponding question for cut-and-project sets has attracted attention.
% through the work of Marklof and Str{\"o}mbergsson, Hammarhjelm, and Gringlaz-Kumar-Weiss. 
We consider simultaneous visibility, which asks what points in a set are visible from every point of another set. This question was studied for the integer lattice by Rearick and Rumsey, among others.

In this work, we investigate simultaneous visibility in cut-and-project sets associated with real quadratic fields satisfying the Hammarhjelm condition. Under suitable geometric conditions on the finite set, we prove that the density of the simultaneously visible points exists and derive an explicit formula for it. We also examine simultaneous visibility in more general cut-and-project sets under a condition on the associated lattice.
\end{abstract}

\maketitle

\section{Introduction}

Let $\Lambda\subset\R^d$ be a locally finite point set, and let
$S\subset\R^d$ be a set. We denote by $\Lambda_{\vis,S}$ the set of
points in $\Lambda\setminus S$ that are visible from every point
$\vs\in S$. Geometrically, these are the points that can be reached
simultaneously from all points of $S$ by straight line segments whose interiors
do not pass through any other point of $\Lambda$:
\[
\Lambda_{\vis,S}
=
\left\{
\vy\in\Lambda\setminus S :
\forall \vs\in S,\ \forall t\in(0,1),\
t\vs+(1-t)\vy\notin\Lambda
\right\}.
\]

Let $\mathcal A=\{A_T\}_{T\in \mathcal{I}}$  be an increasing family of sets whose union
is $\R^d$, where  $\mathcal{I} \subset (0,\infty)$ is unbounded. 
The asymptotic density of $\Lambda$ along $\mathcal A$ is defined by
\[
\theta(\Lambda)\df
\lim_{T\to\infty}
\frac{\card{\Lambda\medcap A_T}}{\vol(A_T)},
\]
provided the limit exists.

The integer lattice $\Lambda=\Z^k$, with $k\geq 2$, is the classical example.
The set $\Z^k_{\vis,\{\zero\}}$ of lattice points visible from the origin is
usually called the set of primitive points. It is well known that its density is
\[
\theta\left(\Z^k_{\vis,\{\zero\}}\right)=\frac{1}{\zeta(k)}
\]
for standard families of sets $\mathcal{A}$ such as cubes and balls. See for example \cite{baake_diffraction_2000} and its appendix for a discussion of more general families. It was shown in \cite{erdos1958elementary} that the set $\mathbb{Z}^k$ contains arbitrarily large cubes containing no visible points from the origin (see also \cite[Proposition~4]{baake_diffraction_2000}).

The set $\mathbb{Z}_{\vis,\{\zero\}}^k$ is  uniformly discrete but it is not relatively dense and hence is not a Delone set (see  \Cref{subsection:cut_and_project_sets} for the definition). In fact, since these holes repeat periodically throughout the lattice, adjoining any set of zero density to $\mathbb{Z}_{\vis,\{\zero\}}^k$ still does not produce a Delone set.

\subsection{Simultaneous visibility in the integer lattice}
Denote by $\cP$ the set of prime numbers and let $S \subset \Z^k$. For each $p\in\cP$, let
\[
\pi_p: \Z^k
\longrightarrow 
\bigl(\quotient{\Z}{p\Z}\bigr)^k,
\qquad
(x_1,\dots,x_k)
\longmapsto
(x_1\bmod p,\dots,x_k\bmod p)
\]
be the coordinatewise projection. We write
\[
s(p)\df\card{\pi_p(S)}
\]
for the number of distinct residue classes modulo $p$ represented by the
elements of $S$.

Rumsey~\cite[Theorem~7]{rumsey_sets_1966} proved that, with respect to the  family of cubes $[-n,n]^k$, the density of $\Z^k_{\vis,S}$ exists and is given by
\begin{equation*}
\theta\left(\Z^k_{\vis,S}\right)
=
\prod_{p\in\cP}
\left(1-\frac{s(p)}{p^k}\right).
\end{equation*}
The special case in which $S$ is finite and its points are pairwise visible had been studied earlier by Rearick~\cite{rearick_mutually_1966}, who established the same formula with an effective error term. Rumsey observed that the same density formula holds for a family of boxes
\[
\left\{
[-n_1,n_1]\times\cdots\times[-n_k,n_k]: \min\{n_1,\ldots,n_k\} \to \infty \right\},
\]
and extended
the formula for certain infinite subsets $S$ satisfying appropriate hypotheses.
For an arbitrary finite set $S$, density along a family of Euclidean balls follows from \cite[Theorem~1.3]{kumar_simultaneous_2026}. We
refer the reader to \cite{berend_simultaneous_2023, athreya_visibility_2023, pleasants_entropy_2013, baake_diffraction_2000, araujo2026sarnak_thesis} and the references therein for further details on simultaneous visibility and related results.

\subsection{Simultaneous visibility in cut-and-project sets}

Let $n=d+m$, where $d,m \in \{1,2,\ldots\}$, and
consider the decomposition
\begin{equation}
\label{eq:direct_sum_decomposition1}
\R^n= \R^d\oplus \R^m.
\end{equation} 
Let $\cL$ be a lattice in $\mathbb{R}^n$, and denote by $\pi_{\textup{phys}}$ and $\pi_{\textup{int}}$ the
projections onto the first $d$ coordinates and the last $m$ coordinates, respectively.
Let $\cW\subset \R^m$ be a bounded measurable subset, called the \emph{window}. We assume that $\pi_{\internal}(\cL)$ is dense in $\R^m$, $\pi_{{\phys}}|_{\cL}$ is injective. The associated cut-and-project set is defined by
\begin{equation}
\label{eq:euclidean_cut_and_project_set}
\Lambda(\cW,\cL) 
\df 
\{\pi_{{\phys}}(\vx) : \vx \in \cL, \, \pi_{{\internal}}(\vx) \in \cW\}.
\end{equation}
We say that $\Lambda=\Lambda(\cW,\cL)$ is \emph{irreducible} if the window $\cW$ is bounded, has nonempty interior, and has boundary of Lebesgue measure zero. Let $D \subset \R^d$ be a Jordan measurable set with $\vol(D)>0$. The density of an irreducible cut-and-project set $\Lambda$ along the increasing family $\{TD\}_{T>0}$, where 
\[
TD \df \{t \vx : \vx\in D,\, t\in [0,T]\}
\]
exists and is given by
\[
\theta(\Lambda(\cW,\cL))= \frac{\vol(\cW)}{\covol(\cL)},
\]
see for example \cite[Proposition 3.2]{marklof_free_2014}. In \Cref{subsection:cut_and_project_sets} we define cut-and-project sets in greater generality.

Marklof and Str{\"o}mbergsson~\cite[Theorem~1]{marklof_visibility_2015} proved that, for any
irreducible cut-and-project set \eqref{eq:euclidean_cut_and_project_set}, the density of
$\Lambda_{\vis,\{\zero\}}$ exists and satisfies
\begin{equation}
\label{eq:MS_inequalities}
0<\theta(\Lambda_{\vis,\{\zero\}})\leq \theta(\Lambda).
\end{equation}

Hammarhjelm~\cite{hammarhjelm_density_2022} gave examples of cut-and-project sets
$\Lambda(\cW,\cL)\subset\R^2$ for which the inequality in
\eqref{eq:MS_inequalities} is strict and computed
$\theta(\Lambda_{\vis,\{\zero\}})$ explicitly. 
The constructions in \cite{hammarhjelm_density_2022} include the
Ammann--Beenker point set and the vertex set of the Penrose tiling. Some of
these examples were also considered by Sing~\cite{sing_visible_nodate}. These
constructions are based on algebraic lattices over real quadratic fields. 

To state our results, we first recall the necessary notions from algebraic number theory.
Let $K=\bbQ(\sqrt \fd)$ be a real quadratic field, where $\fd \geq 2$ is a 
square-free integer, and let $\cO_K$ be its ring of integers. We assume throughout that
$\cO_K$ is a principal ideal domain. The norm of an element $x \in \mathcal{O}_K$ is defined by
\[
N(x)= x\sigma(x),
\]
where $\sigma$ is the
nontrivial Galois automorphism of $K$. The \emph{unit group} 
$\mathcal{O}_K^{\times}$ consists of the elements 
$x \in \mathcal{O}_K$ such that $N(x)=\pm 1$. The unique element $\lambda = \lambda_K
\in \mathcal{O}_K^{\times}$ such that 
$$\mathcal{O}_K^{\times}= \{\pm 1\}\times \{\lambda^i\, :\, i\in
\Z\},\qquad \lambda>1,$$ 
is called the \emph{fundamental unit} of $K$.
Let $I \subseteq \cO_K$ be a nontrivial ideal. The \emph{norm} of $I$ is defined by
\[
N(I) \df \card{\quotient{\cO_K}{I}},
\]
which is always finite. Moreover, $N(I)= |N(x)|$, where $I$ is generated by $x$.
The Dedekind zeta function of $K$ is then defined by
\[
\zeta_{\cO_K}(s) \df \sum_{\substack{I\subseteq\cO_K\\I\neq(0)}}\frac{1}{N(I)^s},
\qquad
\Re(s)>1.
\]
We refer the reader to \cite{marcus_number_2018} for further details.

The Minkowski map
\[
f:K\hookrightarrow\R^2,
\qquad
\alpha\mapsto(\alpha,\sigma(\alpha)),
\]
embeds $\cO_K$ as a lattice $f(\cO_K)$ in $\R^2$.
We say that $K$
satisfies the \emph{Hammarhjelm condition} if
\begin{equation}\label{eq: def_of_H_condition}
f(\cO_K)\medcap ((1,\lambda)\times[-1,1])=\emptyset.    
\end{equation}
For $2\leq \fd \leq 100$, the Hammarhjelm condition is satisfied only for
\[
\fd=2,5,13,29,53;
\]
see \cite{gringlaz_error_2025,hammarhjelm_density_2022}. Recently,
Rudnick~\cite{rudnick_classification_2026} completed the classification by proving that the real quadratic
fields satisfying the Hammarhjelm condition are precisely those with
\[
\fd=2,5,13,29,53,173,293.
\]
In  \Cref{lemma:lower_bound_on_conjugates_of_primes}, we strengthen the Hammarhjelm condition \eqref{eq: def_of_H_condition} by proving that, for each of these fields, 
\[
f(\cO_K)\medcap \left((1,\lambda)\times\left[-\left(1+\frac{1}{\lambda}\right),\left(1+\frac{1}{\lambda}\right)\right]\right)=\emptyset,
\]
which is crucial for the geometric condition on the finite set $S\subset \Lambda(\cW,\cL)$ in  \Cref{thm:visibility_from_set_S}.

By abuse of notation, for
$\vx=(x_1,\ldots,x_d)\in K^d$ we write
\[
\sigma(\vx)
\df
(\sigma(x_1),\ldots,\sigma(x_d)).
\]
For every $d\in\N$,
\begin{equation} \label{eq:def_lattice}
\cL
\df
\{(\vx,\sigma(\vx)):\vx\in\cO_K^d\}
\end{equation}
is a lattice in $\R^{2d}$.
Let $d\geq 2$, let $\cL$ be the
lattice in \eqref{eq:def_lattice}, and $\cW\subset\R^d$ be a bounded, convex, centrally symmetric window with nonempty interior. Suppose that the real quadratic field $K$ satisfies the Hammarhjelm condition. For the
associated cut-and-project set $\Lambda=\Lambda(\cW,\cL)$, the density of the
set of visible points from the origin, along the family $\{TD\}_{T>0}$ exists, and is given
by
\begin{equation}\label{eq:density_for_origin}
\theta(\Lambda_{\vis,\{\zero\}})
=
\left(1-\frac{1}{\lambda^d}\right)
\frac{\theta(\Lambda)}{\zeta_{\cO_K}(d)},    
\end{equation}
where $D$ is a Jordan measurable set with $\vol(D)>0$, see \cite[Theorems~4.9, 4.12, 4.15]{hammarhjelm_density_2022}. An effective rate for the density is given in  \cite[Theorem~5.1]{gringlaz_error_2025} when $D$ is a bounded convex set with $\vol(D)>0$. 

According to \cite[Theorem~1.1]{gringlaz_error_2025}, $\Lambda_{\vis,\{\zero\}}$ contains arbitrarily large holes. In fact, it follows from the proof of \cite[Theorem~1.3]{kumar_diffraction_2026}, together with \cite[Proposition~2.12]{huck_pattern_2015}, that these holes repeat. Adjoining any zero density set to $\Lambda_{\vis,\{\zero\}}$ still does not produce a Delone set.

Let $S\subset\Lambda$ be a  finite set and let $\cP_{\cO_K}$ be the set of nonzero
prime ideals of $\cO_K$. For each $\fp\in\cP_{\cO_K}$, denote by
\[
\pi_{\fp}:\cO_K^d\to\left(\quotient{\cO_K}{\fp}\right)^d
\]
the quotient map, and define
\[
s(\fp)\df\card{\pi_{\fp}(S)}.
\]
For $\vs\in\Lambda$, define
\begin{equation*} %\label{eq:window_W_s}
\cW_{\vs}
\df
\frac{1}{\lambda}\bigl(\cW+(\lambda+1)\sigma(\vs)\bigr).
\end{equation*}
For a nonempty, finite set $I\subset\Lambda$, denote
\[
\cW_I\df\bigcap_{\vs\in I}\cW_{\vs}.
\]
If $\cW$ is convex, then each $\cW_I$ is convex. 
Furthermore, if $\cW$ is bounded, centrally symmetric, has nonempty interior, and  for every $\vs \in S$ we have the geometric condition
\begin{equation}
\label{eq:geometric_condition}
-(\cW-\sigma(\vs))
\subset
\left(1+\frac{1}{\lambda}\right)(\cW-\sigma(\vs)),    
\end{equation}
then (see \Cref{lemma:W_s_is_contained_in_W})
\[
\cW_{\vs}\subset \cW.
\]

In the following \namecref{thm:visibility_from_set_S}, we prove that the density of the set of simultaneously visible points from a finite set $S$ along $\{TD\}_{T>0}$ exists and derive an explicit formula for it, assuming that each $\vs \in S$ satisfies the condition \eqref{eq:geometric_condition}.

\begin{thm}
\label{thm:visibility_from_set_S}
Let $K=\bbQ(\sqrt \fd)$ be a real quadratic field with ring of integers
$\cO_K$ satisfying the Hammarhjelm condition. Let $d\geq 2$, let $\cL$ be the
lattice defined in \eqref{eq:def_lattice}, and let $\cW\subset\R^d$ be a bounded, convex, centrally symmetric window with nonempty interior. Let
$\Lambda=\Lambda(\cW,\cL)$ be the associated cut-and-project set. Suppose that
$S\subset\Lambda$ is finite and that \eqref{eq:geometric_condition} holds for every $\vs\in S$.
Then the density of $\Lambda_{\vis,S}$ along $\{TD\}_{T>0}$ exists and
is given by
\begin{equation*}
\theta(\Lambda_{\vis,S})
=
\dens(\cL)
\left(
\vol(\cW)
-
\sum_{\emptyset\neq I\subseteq S}
(-1)^{\card{I}+1}
\vol(\cW_I)
\right)
\prod_{\fp\in\cP_{\cO_K}}
\left(1-\frac{s(\fp)}{N(\fp)^d}\right).
\end{equation*}
Here $D\subset \mathbb{R}^d$ is a bounded convex set with $\vol(D)>0$.
\end{thm}

It has been shown in \cite{baake_diffraction_2000} that, for $k\geq 2$, the set $\mathbb{Z}_{\vis,\{\zero\}}^k$ is a weak model set of maximal density; see \Cref{sec: Weak_model_set} for the definition. Moreover, it has been shown in \cite[Theorem 1.1]{kumar_diffraction_2026} that if $\mathbb{Z}_{\vis, S}^k$ has positive density along $[-L, L]^k$, then  $\mathbb{Z}_{\vis, S}^k$ is a weak model set of maximal density. Furthermore, it follows from the proof of  \cite[Theorem 1.3]{kumar_diffraction_2026} that the set $\Lambda_{\vis,\{\zero\}}$ is also a weak model set of maximal density along the sequence $\{TD\}_{T>0}$. The following \namecref{thm_weak_model} extends this result to the set $\Lambda_{\vis, S}$.

\begin{corollary}
\label{thm_weak_model}
Under the hypotheses of \Cref{thm:visibility_from_set_S}, suppose that the density $\theta(\Lambda_{\vis,S})$ is positive along $\{TD\}_{T>0}$. Then the set $\Lambda_{\vis, S}$ is a weak model set of maximal density.
\end{corollary}

Denote by
\[
\scrX_n\df\SL_n(\R)/\SL_n(\Z),
\qquad
\scrY_n\df\ASL_n(\R)/\ASL_n(\Z),
\]
the quotient spaces that parameterize covolume-one lattices and affine lattices in $\R^n$,
respectively. Denote by 
$\haar_{\scrX_n}$ and $\haar_{\scrY_n}$ the probability measures on 
$\scrX_n$
 and $\scrY_n$ induced by Haar measures on $\SL_n(\R)$ and $\ASL_n(\R)$, respectively.
Fix the decomposition in \eqref{eq:direct_sum_decomposition1} and a bounded measurable window
$\cW\subset\R^m$ 
with nonempty interior and 
whose boundary has Lebesgue measure
zero. The map
\[
\Psi(\cL)\df\Lambda(\cW,\cL)
\]
associates a cut-and-project set with each lattice or affine lattice. We denote by
\[
\mu\df\Psi_{\ast}\haar_{\scrX_n},
\qquad
\overline{\mu}\df\Psi_{\ast}\haar_{\scrY_n}
\]
the resulting push-forward measures on the space of closed subsets of
$\R^d$. These are, respectively, linear and affine
Ratner--Marklof--Str{\"o}mbergsson measures first considered in \cite{marklof_visibility_2015} and then classified in \cite{ruhr_classification_2024}. We review the
general definition of RMS measures in
\Cref{subsection:RMS_measures}.

Let $\cL\subset \R^n$ be a lattice, and let $\tilde{S} \subset \cL$. For each $p\in \mathcal{P}$, let
$$\pi_{\cL,p}: \cL\to \quotient{\cL}{p\cL},$$
be the natural projection. We write
$$s_{\cL}(p) \df \card{\pi_{\cL,p}(\tilde{S})}.$$

\begin{thm}
\label{thm:generic_density_simul_vis}
Let $n=d+m$, where $d\geq2$ and $m\geq1$, and let
$\cW\subset\R^m$ be a bounded, convex window with nonempty interior and boundary of Lebesgue measure
zero.
Let $\Lambda=\Lambda(\cW,\cL)$ be an irreducible cut-and-project set.
Suppose that
\begin{equation}
\label{eq:condition_on_lattice}
\dim\left(
\operatorname{span}_{\R}\{\vv_1,\vv_2\}
\medcap
(\{\zero\}\times\R^m)
\right)=0
\end{equation}
for every pair of linearly independent vectors
$\vv_1,\vv_2\in\cL$. Then, for every finite set $S\subset\Lambda$, the
density of $\Lambda_{\vis,S}$ along $\{TD\}_{T>0}$ exists and is given by
\[
\theta(\Lambda_{\vis,S})
=
\frac{\vol(\cW)}{\covol(\cL)}
\prod_{p\in\cP}
\left(1-\frac{s_{\cL}(p)}{p^n}\right),
\]
where $D\subset\R^d$ is bounded and convex with positive volume,
$\tilde S \df (\pi_{\phys}|_{\cL})^{-1}(S)$.
Moreover, condition \eqref{eq:condition_on_lattice} holds for
$\mu$-almost every irreducible cut-and-project set. Consequently, the formula above
holds for $\mu$-almost every $\Lambda(\cW,\cL)$ and every finite
$S\subset\Lambda$.
\end{thm}
\begin{remark}
In \cite[Theorem~3.4]{gringlaz_error_2025}, for $S=\{\zero\}$, the density formula in the
preceding \namecref{thm:generic_density_simul_vis} was  obtained for $\mu$-almost every cut-and-project set together with an error term. 
\end{remark}
We prove \Cref{thm:generic_density_simul_vis} in
\Cref{section:random}.
\begin{prop}
\label{prop:equality_affine_case}
Let $d\geq2$, let
$\cW \subset \R^m$ be a bounded window with nonempty interior and boundary of Lebesgue measure
zero, and let $S\subset\R^d$ be countable. Then, for
$\overline{\mu}$-almost every cut-and-project set $\Lambda(\cW,\cL)$,
\[
\Lambda(\cW,\cL)\setminus S
=
\Lambda(\cW,\cL)_{\vis,S}.
\]
Therefore,
\[
\theta\left(\Lambda(\cW,\cL)_{\vis,S}\right)
=
\frac{\vol(\cW)}{\covol(\cL)}.
\]
\end{prop}

\subsection{Remarks}

\begin{enumerate}
\item 
During the preparation of this manuscript, the authors became
aware of recent work by H.~Shamir~\cite{hayim_2026} that generalizes
the density formula \eqref{eq:density_for_origin}. In particular,
Shamir avoids the Hammarhjelm condition and proves the result for every
real quadratic field $K=\bbQ(\sqrt{\fd})$, where $\fd\geq2$ is a
square-free integer.

\item 
For regular cut-and-project sets, we refer the reader to \cite[Theorem 7]{baake_dynamical_2004}, where it is proved that, the associated dynamical system with a translation bounded invariant probability measure, pure point diffraction and pure point dynamical spectrum are equivalent.  For further details on diffraction, see \cite{Baake1,Martin,Hof2,kumar_diffraction_2026} and the references therein.

\item 
Suppose that $\theta(\Lambda_{\vis,S})>0$, it follows from  \Cref{thm_weak_model} and \cite[Theorem 7]{baake_weak_2017} that the diffraction spectrum of $\Lambda_{\vis, S}$ is pure point and translation bounded. We refer the reader to \cite[Corollary 18]{baake_weak_2017} for the relation between pure point diffraction of the weak model sets with maximal density and the spectrum of associated dynamical system.

\item We stated \Cref{thm:generic_density_simul_vis} and \Cref{prop:equality_affine_case} for the measures $\mu$ and $\overline{\mu}$, respectively. However, extending these results with respect to arbitrary RMS  measures remains an open problem.
\end{enumerate}

\subsection*{Acknowledgments} 

The authors thank Barak Weiss, Michael Baake, and Zeev Rudnick for many helpful discussions and comments.
The first author was supported by ISF Grant 2860/24. The second author was supported by ISF Grants 2021/24 and \#674/24.

\section{Preliminaries}

\subsection{Averaging sequences}
\label{subsec:van_Hove_sequences}

Let $G$ be a $\sigma$-compact locally compact abelian group equipped
with a Haar measure $\theta_G$. An \emph{averaging sequence} is a
family $\{A_T\}_{T\in\mathcal I}$ of precompact Borel subsets of
$G$, indexed by an unbounded set
$\mathcal I\subset(0,\infty)$, such that
\begin{enumerate}
\item for  $T<T'$, $A_T \subseteq A_{T'}$,
\item$\lim_{T\to\infty}\theta_G(A_T) = \theta_G(G)$, and
\item $\bigcup_{T \in \mathcal{I}} A_T = G$ up to measure zero.
\end{enumerate}
An averaging sequence is a \emph{van Hove sequence} if $\theta_G(A_T) > 0$ for all sufficiently large $T$, and for every compact set $K \subseteq G$, the $K$-boundary is asymptotically negligible, that is,
\[
\lim_{T\to\infty}
\frac{\theta_G\bigl(\partial^K A_T\bigr)}{\theta_G(A_T)}
=
0,
\]
where
\[
\partial^K A_T
\df
\bigl((K+A_T)\setminus \interior{A}_T\bigr)
\cup
\bigl((-K+\overline{A_T^{\mathsf{c}}})\medcap A_T\bigr)
\]
is the $K$-boundary of $A_T$. 
Here $A_T^{\mathsf{c}}$, $\interior{A_T}$ and $\overline{A_T}$ denote the complement, interior and closure of $A_T \subset G$, respectively, see \cite[p.~125]{moody_uniform_2002}, \cite[p.~145]{schlottmann_cut-and-project_1998}, and \cite[p.746]{huck_pattern_2015}. The existence of van Hove sequences in $\sigma$-compact LCAGs is shown in \cite{schlottmann_cut-and-project_1998}.
In the Euclidean setting, namely $G=\R^d$, standard examples of van Hove sequences
include increasing sequences of Euclidean balls and cubes. More generally,
truncated cones $TA=\{t\vx:\vx\in A,\ t \in[0,T]\}$ of a bounded Jordan set $A\subset\R^d$ with nonempty interior.
We shall use the fact that every bounded convex subset of
$\R^d$ has boundary of Lebesgue measure zero and therefore Jordan measurable \cite{szabo_simple_1997}.

The upper and lower densities of $\Lambda \subset G$ with respect to an averaging sequence $\{A_T\}_{T \in \mathcal{I}}$ are
defined by
\[
\overline{\theta}(\Lambda)
\df
\limsup_{T\to\infty}
\frac{\card{ \Lambda \medcap A_{T}}}{\theta_G(A_T))}
\qquad
\text{and}
\qquad
\underline{\theta}(\Lambda)
\df
\liminf_{T\to\infty}
\frac{\card{ \Lambda \medcap A_{T}}}{\theta_G(A_T))},
\]
respectively.

\subsection{General cut-and-project sets} \label{subsection:cut_and_project_sets}

Let $G$ and $H$ be $\sigma$-compact, locally compact abelian groups (LCAGs) equipped with Haar measures $\theta_G$ and $\theta_H$, respectively. Throughout, we assume that $H$ is compactly generated. By the Structure Theorem for compactly generated locally compact abelian groups, $H$ is topologically isomorphic to $\R^a \times \mathbb{Z}^b \times \mathbb{K}$,
for some integers $a,b \geq 0$ and some compact abelian group $\mathbb{K}$, see \cite[Theorem~9.8]{hewitt_abstract_1979}.

We refer to $G$ as the \emph{physical space} and to $H$ as the \emph{internal space}. Let
$\cL\subset G\times H$ be a lattice, and let $\cW\subset H$ be a measurable
window. 
Let
\[
\pi_{\phys}:G\times H\to G,
\qquad
\pi_{\internal}:G\times H\to H
\]
denote the natural projections. We assume that the restriction $\pi_{\phys}|_{\cL} : \cL \to G$ is injective and that the image $\pi_{\internal}(\cL)$ is dense in $H$.
The associated cut-and-project set is
\[
\Lambda(\cW,\cL)
\df
\{\pi_{\phys}(\vz):\vz\in\cL,\ \pi_{\internal}(\vz)\in\cW\}.
\]
The cut-and-project set $\Lambda(\cW,\cL)$ is called a \emph{model set} if the window $\cW$ is relatively compact and has nonempty interior, 
see \cite{baake_weak_2017} and the references therein.
If, in addition, $\cW = \overline{\interior{\cW}}$ and $\cW$ is compact, then the window is called \emph{proper}.
If $\cW$ is a proper window, and satisfies
\[
\theta_H(\partial \cW)=0,
\]
then the window $\cW$ is called {\it regular} and the cut-and-project set $\Lambda(\cW,\cL)$ is called \emph{irreducible}. 

The density of an irreducible cut-and-project set $\Lambda(\cW,\cL)$ along  a van Hove sequence 
$\{A_n\}_{n \in \mathbb{N}}$ 
exists and is given by
\begin{equation}
\label{eq:density_irreducible_CPS}
\theta(\Lambda(\cW,\cL))
=
\lim_{n\to\infty}
\frac{\card{\Lambda(\cW,\cL)\medcap A_n}}{\theta_G(A_n)}
=
\frac{\theta_H(\cW)}{\covol{(\cL)}},
\end{equation}
see \cite[Theorem~1]{moody_uniform_2002}, \cite{schlottmann_cut-and-project_1998}, \cite[Lemma~3.2, Remark~3.3]{huck_pattern_2015}, \cite[Proposition~3.2]{marklof_free_2014}, and the references therein. Moreover, the density of  irreducible cut-and-project sets is independent of the van Hove sequence, see \cite[Remark 3.3(i)]{huck_pattern_2015}.  Irreducible cut-and-project sets are Delone sets, see \cite[p.~744]{huck_pattern_2015}.

Recall that a set $\mathcal{D} \subset G$ is \emph{uniformly discrete} if there exists an open neighborhood $U $ of the identity in $G$ such that any of its translates $x+U$, with $x \in G$, contains at most one point of $\mathcal{D}$. The set $\mathcal{D} \subset G$ is \emph{relatively dense} if there exists a compact set $K \subset G$ such that $\mathcal{D}+K = G$. A subset $\mathcal{D} \subset G$ is called a \emph{Delone set} if it is both uniformly discrete and relatively dense.

In \Cref{thm:visibility_from_set_S}, the window
$\cW\subset\R^m$ is assumed to be bounded and convex, to have nonempty
interior, and to be centrally symmetric about the origin. The latter
condition means that
\[
\cW=-\cW.
\]

\subsection{Weak model sets of maximal density}\label{sec: Weak_model_set} Let $G$ and $H$ be as in \Cref{subsection:cut_and_project_sets}. A cut-and-project set $\Lambda(\cW,\cL)$ is called a \emph{weak model set} if the window $\cW$ is relatively compact and satisfies
\[
\theta_H(\overline{\cW}) > 0.
\]
Unlike model sets, we do not require the window to have a nonempty interior. Weak model sets are uniformly discrete (see \cite[Proposition 2.4]{huck_pattern_2015}).
Furthermore, if $\cW$ is nowhere dense, then the cut-and-project set $\Lambda(\cW,\cL)$ contains arbitrarily large holes and has the hole-repetitive property, see \cite[Proposition~2.12]{huck_pattern_2015} for the definition and proof of this. 
Consequently, it is not relatively dense.

We say that $\Lambda(\cW,\cL)$ is a \emph{weak model set of maximal density} relative to a given averaging sequence $\mathcal{A}$ if $\Lambda(\cW,\cL)$ is a weak model set, and if the density of $\Lambda(\cW,\cL)$ relative to $\mathcal{A}$ exists and is given by
\[
\theta\bigl(\Lambda(\cW,\cL)\bigr)
=
\frac{\theta_H(\overline{\cW})}{\covol(\cL)}.
\]
We refer the reader to \cite{baake_weak_2017} and the reference therein for further details.

\subsection{Ratner--Marklof--Str{\"o}mbergsson measures}
\label{subsection:RMS_measures}
Let $\SL_n(\R)$ be the group of linear transformations of $\R^n$ that are volume and orientation preserving. Let $\ASL_n(\R) = \SL_n(\R) \ltimes \R^n $ be the group of affine transformations on $\R^n$ that are volume and orientation preserving. Consider the homogeneous spaces
\[
\scrX_n \df \SL_n(\R)/\SL_n(\Z) \quad \text{and} \quad \scrY_n \df \ASL_n(\R)/\ASL_n(\Z)
\]
both equipped with their quotient topologies. These spaces parameterize, respectively, covolume-one lattices and affine lattices in $\R^n$. We write $\haar_{\scrX_n}$ and $\haar_{\scrY_n}$ for the corresponding Haar-Siegel measures. These are the unique probability measures invariant under the actions of $\SL_n(\R)$ and $\ASL_n(\R)$, respectively.

Let $\Cl(\R^n)$ be the collection of closed subsets of $\R^n$. Equipped with the \emph{Chabauty-Fell topology}, $\Cl(\R^n)$ is compact, see \cite[\S~2.2]{ruhr_classification_2024} and the references therein.
After identifying  $\scrX_n$ and $\scrY_n$ as subsets of $\Cl(\R^n)$, the restriction of the Chabauty-Fell topology to these, coincides with their quotient topology. Therefore, lattices and affine lattices satisfy natural inclusions
$$\scrX_n \subset \scrY_n \subset \Cl(\R^n)$$
as topological subspaces.

The natural actions of $\ASL_d(\R)$ and $\SL_d(\R)$ on $\mathbb{R}^d$ induce corresponding actions on $\Cl(\R^d)$. A probability measure on $\Cl(\R^d)$ is called a \emph{Ratner-Marklof-Str{\"o}mbergsson linear measure}, or an \emph{RMS linear measure} for short, if it is invariant and ergodic with respect to the $\SL_d(\R)$-action and gives full measure to the set of irreducible cut-and-project sets.

Identifying the group $\SL_d(\R)$ as  a subgroup of $\ASL_d(\R)$ in the canonical way, an RMS linear measure does not need to be invariant under the $\ASL_d(\R)$ action. When this additional invariance property  holds, the measure is called an \emph{RMS affine measure}.

Fix the decomposition in \eqref{eq:direct_sum_decomposition1}, the associated projections $\pi_{\phys}$ and $\pi_{\internal}$, a
bounded measurable window $\cW\subset\R^m$   with 
nonempty interior and
boundary of Lebesgue measure zero. 
Define
\[
\Psi_{\cW}:\scrY_n\longrightarrow\Cl(\R^d),
\qquad
\Psi_{\cW}(\cL)
\df
\Lambda(\cW,\cL).
\]
The map $\Psi_{\cW}$ is Borel measurable (\cite[Proposition 2.1]{ruhr_classification_2024}), and hence we
may define the probability measures
\[
\mu
\df
(\Psi_{\cW})_*\haar_{\scrX_n},
\qquad
\overline{\mu}
\df
(\Psi_{\cW})_*\haar_{\scrY_n}
\]
on $\Cl(\R^d)$. It was proved that the measures $\mu$ and $\overline{\mu}$ are RMS linear and affine measures, respectively, 
see \cite[Theorem 4.1]{ruhr_classification_2024}.
We refer the reader to \cite{marklof_free_2014, ruhr_classification_2024} for further details on RMS measures. 

For a lattice $\cL\subset\R^d\oplus\R^m$ and a pair of linearly
independent vectors $\vv_1,\vv_2\in\cL$, define
\begin{equation*}
% \label{eq:condition_on_lattice}
\dim_{\cL}(\vv_1,\vv_2)
\df
\dim\left(
\operatorname{span}_{\R}\{\vv_1,\vv_2\}
\medcap
(\{\zero\}\times\R^m)
\right).
\end{equation*}
This quantity can be calculated for any given
lattice. Moreover, for almost every $g \in \GL_n(\R)$, the lattice
$\cL=g\Z^n$ and every pair of linearly independent vectors satisfies $\dim_{\cL}(v_1,v_2)=0$,
see
the proof of \cite[Proposition~3.3]{gringlaz_error_2025}.
Thus, any statement proved under condition
\eqref{eq:condition_on_lattice} applies to
$\mu$-almost every irreducible cut-and-project set.

\subsection{Visibility from a point in a cut-and-project set}

Let $K=\bbQ(\sqrt \fd)$ be a real quadratic field with ring of integers
$\cO_K$ satisfying the Hammarhjelm condition. For two distinct points $\vx=(x_1,\dots,x_k)$ and $\vy=(y_1, \dots, y_k)$ in $\mathcal{O}_K^k$, we define
\[
\gcd(\vx -\vy )=\gcd\big((x_1-y_1),\ldots,(x_k-y_k)\big),
\]
where $(x_i-y_i)$ denotes the ideal generated by $x_i-y_i$. Since $\mathcal{O}_K$ is a principal ideal domain, every such ideal is principal. By abuse of notation, we write $\gcd(\vx-\vy)$ for a generator of this ideal. In particular, if 
$$\gcd(\vx-\vy)= \mathcal{O}_K,$$
we simply write
$$\gcd(\vx-\vy)= 1.$$
We also define
$$\gcd(\zero)\df0.$$

Let $\cW\subset\R^d$ be a bounded, 
convex, centrally symmetric window with nonempty interior, and let $\cL$ be the lattice defined in
\eqref{eq:def_lattice}. By \cite[Lemma 5.6]{gringlaz_error_2025},  
\begin{equation*} 
\Lambda_{\vis,\{\zero\}}(\cW,\cL)
= \left\{ \vx  \in \Lambda(\cW, \cL)\setminus\{\zero\} :
\gcd(\vx) = 1, \; \sigma(\vx) \notin \frac{1}{\lambda}\cW \right\}.
\end{equation*}
The following \namecref{lemma:visibility_from_a_point_s} extends this characterization to visibility from an arbitrary point $\vs\in \Lambda(\cW,\cL)$ satisfying a suitable geometric condition.
\begin{lemma} \label{lemma:visibility_from_a_point_s}
Under the hypotheses of Theorem \ref{thm:visibility_from_set_S}, 
let $\vs\in\Lambda$ satisfy the geometric condition \eqref{eq:geometric_condition}.
Then the set of points visible from $\vs$ is given by
\begin{equation} \label{eq:visible_from_point_s}
\Lambda_{\vis,\{\vs\}}
=
\left\{
\vx\in\Lambda:
\gcd(\vx-\vs)=1,\
\sigma\!\left(\frac{\vx-\vs}{\lambda}\right)\notin\cW-\sigma(\vs)
\right\}.
\end{equation}
\end{lemma}

\section{Proof of \texorpdfstring{\Cref{lemma:visibility_from_a_point_s}}{}}
\label{section:proof_of_first_big_lemma}

Let $K=\bbQ(\sqrt \fd)$ be a real quadratic field satisfying the Hammarhjelm
condition, and let $\lambda>1$ be its fundamental unit. 
Rudnick \cite{rudnick_classification_2026} shows that 
\[
\fd \in
\{2,5,13,29,53,173,293\}.
\]
Observe that if
\[
\fd
\in
\{13,29,53,173,293\},
\]
then
\[
\fd=x^2+4,
\qquad
x\in\{3,5,7,13,17\}.
\]
Writing $x=2m+1$, we have
\[
\cO_K=\Z[\omega],
\qquad
\omega=\frac{1+\sqrt \fd}{2},
\]
and
\[
\lambda=\frac{x+\sqrt \fd}{2}=m+\omega.
\]
In particular,
\[
N(\lambda)=-1,
\qquad
\sigma(\lambda)=-\frac{1}{\lambda},
\]
and
\[
\frac{1}{\lambda}=\frac{\sqrt \fd-x}{2},
\qquad
1+\frac{1}{\lambda}=\frac{\sqrt \fd -x+2}{2}.
\]

\begin{lemma}
\label{lemma:inequality_for_fundamental_unit}
Let $K=\bbQ(\sqrt \fd)$ be a real quadratic field and let $\lambda$ be its fundamental unit. Then
\begin{equation}\label{eq: condition_on_lambda}
1+\frac{1}{\lambda}\leq \lambda.    
\end{equation}
Equality occurs only for $K=\bbQ(\sqrt5)$.
\end{lemma}

\begin{proof}
It can easily be seen that the condition \eqref{eq: condition_on_lambda} is equivalent to 
$$\lambda \geq \frac{1+\sqrt{5}}{2}.$$
\medskip
\noindent \textbf{Case 1:} $\fd=2,3\, {\rm mod}\, 4$.

In this case, the ring of integers is $\mathbb{Z}[\sqrt{\fd}]$. We claim that the fundamental unit $\lambda= m+ n\sqrt{\fd}$ for some $n,m\in \mathbb{N}$. In fact, if $n=0$, then $N(\lambda)= m^2=\pm1$, so $m=\pm 1$, which is contradiction of $\lambda>1$. Now if $n<0$, then $\sigma(\lambda)= m-n\sqrt{\fd}>m+n\sqrt{\fd}= \lambda>1$, which contradicts $|N(\lambda)|= \lambda\sigma(\lambda)=1$.

Now, if $m=0$ then we have $N(\lambda)= -n^2\fd\neq \pm 1$. So we must have $m\neq 0$. Suppose that $m<0$, then $\sigma(\lambda)= m-n\sqrt{\fd}<0$. Therefore
\begin{equation*}
\begin{split}
 \lambda &= \frac{1}{|\sigma(\lambda)|}= \frac{1}{n\sqrt{\fd}-m} \leq \frac{1}{n\sqrt{\fd}}  <1,
\end{split}    
\end{equation*}
which contradicts $\lambda>1$.

Since $n,m\geq 1$, we have
$$\lambda= m+n\sqrt{\fd}\geq 1+\sqrt{\fd}\geq 1+\sqrt{2}>\frac{1+\sqrt{5}}{2}$$

\medskip
\noindent \textbf{Case 2:} $\fd=1\, {\rm mod}\, 4$.

In this case the ring of integers is $\mathbb{Z}[\omega]$, where $\omega =(1+\sqrt{\fd})/2$. By \cite[Lemma 2.1]{rudnick_classification_2026}, we have
$\lambda= m+ n\omega$ with $m\geq 0$, $n\geq 1$.
Thus
\begin{equation*}
 \lambda\geq 0+ \omega= (1+\sqrt{\fd})/2\geq \frac{1+\sqrt{5}}{2},   
\end{equation*}
because $\fd>1$, and equality if and only if $\fd=5$.
\end{proof}

Denote
\[
\bbP
\df
\{\pi\in\cO_K:\pi \textup{ is prime and } 1<\pi<\lambda\}.
\]
The set $\bbP$ plays the role of a fundamental domain for the action of
the unit group on prime elements of $\cO_K$: after multiplying by a unit, and
possibly by $-1$, every prime element has a representative in $\bbP$.

It follows from \cite[Lemma~4.7]{hammarhjelm_density_2022} that, if $\fd = 2$,
then
\[
|\sigma(\pi)|
\geq
1+\frac{1}{\lambda}
\qquad
\text{for all }
\pi
\in
\bbP,
\]
and if $\fd = 5$, then \cite[Lemma~4.13]{hammarhjelm_density_2022} gives the stronger
bound
\[
|\sigma(\pi)|\geq 2\lambda
\qquad
\text{for all }\pi\in\bbP.
\]

\begin{lemma} \label{lemma:lower_bound_on_conjugates_of_primes}
Let $K=\bbQ(\sqrt \fd)$ be a real quadratic field satisfying the Hammarhjelm
condition. Then, for every $\pi\in\bbP$,
\[
|\sigma(\pi)|\geq 1+\frac{1}{\lambda}.
\]
Moreover, for $\fd \in
\{2,13,29,53,173,293\}$, this bound is sharp.
\end{lemma}

\begin{proof}
By \cite[Lemmas~4.7 and~4.13]{hammarhjelm_density_2022}, the desired lower
bound holds for $\fd = 2,5$. Moreover, for $\fd = 2$, the prime element
$\pi=\sqrt2$ satisfies
\[
1<\sqrt2<\lambda,
\qquad
|\sigma(\sqrt2)|=\sqrt2=1+\frac{1}{\lambda},
\]
so the bound is sharp in this case. It remains to treat
\[
\fd \in \{13,29,53,173,293\}.
\]

As discussed at the beginning of \Cref{section:proof_of_first_big_lemma}, there is
$x\in\{3,5,7,13,17\}$ such that
\[
\fd = x^2+4.
\]
Writing $x=2m+1$, we have
\[
\cO_K=\Z[\omega],
\qquad
\omega=\frac{1+\sqrt \fd}{2},
\qquad
\lambda=m+\omega.
\]
Also,
\[
\sigma(\lambda)=-\frac{1}{\lambda},
\qquad
\sigma(\omega)=-m-\frac{1}{\lambda}.
\]

Let $\pi\in\bbP$. There exist unique integers $a,b\in\Z$ such that
\[
\pi=a+b\omega.
\]
Since $\omega-\sigma(\omega)=\sqrt\fd$, we have
\begin{equation} \label{eq:conjugate-pi}
\sigma(\pi)=\pi-b\sqrt \fd.
\end{equation}

We divide the proof into cases according to $b$.

\begin{itemize}
\item If $b\leq -1$, then by \eqref{eq:conjugate-pi},
\[
\sigma(\pi)=\pi+|b|\sqrt \fd>1+\sqrt \fd>\sqrt{\fd}.
\]
Since $1+\lambda^{-1}<2<\sqrt \fd$, it follows that
\[
|\sigma(\pi)|>1+\frac{1}{\lambda}.
\]

\item If $b=0$, then $\pi=a\in\Z$. Since $\pi>1$, we have $a\geq 2$, and hence
\[
|\sigma(\pi)|=a\geq 2>1+\frac{1}{\lambda}.
\]

\item If $b\geq 2$, then by \eqref{eq:conjugate-pi},
\[
\sigma(\pi)=\pi-b\sqrt \fd\leq \pi-2\sqrt \fd.
\]
Since $\fd=x^2+4$, we have $x<\sqrt \fd$, and therefore
\[
\lambda=\frac{x+\sqrt \fd}{2}<\sqrt \fd.
\]
As $\pi<\lambda$, we get
\[
\sigma(\pi)<\sqrt \fd-2\sqrt \fd=-\sqrt \fd.
\]
Since $\sqrt \fd>2>1+\lambda^{-1}$, it follows that
\[
|\sigma(\pi)|>1+\frac{1}{\lambda}.
\]

\item If $b=1$, then $\pi=a+\omega$. The inequality
$\pi<\lambda=m+\omega$ gives $a\leq m-1$. Hence
\[
\sigma(\pi)
=
a+\sigma(\omega)
=
a-m-\frac{1}{\lambda}
\leq
-1-\frac{1}{\lambda}.
\]
Thus
\[
|\sigma(\pi)|\geq 1+\frac{1}{\lambda}.
\]
\end{itemize}

It remains to prove sharpness for
\[
\fd\in\{13,29,53,173,293\}.
\]
Set
\[
\pi_0=\lambda-1.
\]
Since $\lambda>x\geq 3$, we have
\[
1<\pi_0<\lambda.
\]
Moreover,
\[
|\sigma(\pi_0)|
=
|\sigma(\lambda)-1|
=
\left|-\frac{1}{\lambda}-1\right|
=
1+\frac{1}{\lambda}.
\]
Finally,
\[
N(\lambda-1)=-x.
\]
Since $x\in\{3,5,7,13,17\}$ is a rational prime, the ideal $(\lambda-1)$ has
prime norm. Hence $(\lambda-1)$ is a prime ideal, so $\lambda-1$ is a prime
element. Therefore $\pi_0\in \bbP$, and the bound is sharp.
\end{proof}

\begin{proof}[Proof of \Cref{lemma:visibility_from_a_point_s}]

We first prove the inclusion $\subseteq$ in
\eqref{eq:visible_from_point_s}. Let $\vx\in\Lambda_{\vis,\{\vs\}}$. 
We
show that
\[
\gcd(\vx-\vs)=1
\qquad\text{and}\qquad
\sigma\!\left(\frac{\vx-\vs}{\lambda}\right)\notin \cW-\sigma(\vs).
\]

Suppose first that
\[
\sigma\!\left(\frac{\vx-\vs}{\lambda}\right)\in \cW-\sigma(\vs).
\]
Then
\[
\vy'=\vs+\frac{\vx-\vs}{\lambda}
\]
belongs to $\Lambda$, and since $\lambda>1$, the point $\vy'$ lies in
the interior of the line segment joining $\vs$ and $\vx$. This contradicts
the visibility of $\vx$ from $\vs$.

Next, suppose that $\gcd(\vx-\vs)\neq 1$. Then there exists a prime element
$\pi\in\cO_K$ dividing $\vx-\vs$. After multiplying $\pi$ by a unit, and
possibly by $-1$, we may assume that $\pi\in\bbP$. Set
\[
\vy=\vs+\frac{\vx-\vs}{\pi}.
\]
Then $\vy\in\cO_K^d$, and since $1<\pi<\lambda$, the point $\vy$ lies in
the interior of the line segment joining $\vs$ and $\vx$. 
It remains to
show that 
$\vy\in\Lambda(\cW,\cL)$. We have
\[
\sigma(\vy)
=
\sigma(\vs)+\sigma(\pi)^{-1}\sigma(\vx-\vs),
\]
so it suffices to prove that $\sigma(\vy)\in\cW$.

\medskip
\noindent\textbf{Case 1.} 
Assume $\sigma(\pi) > 0$.

By \Cref{lemma:lower_bound_on_conjugates_of_primes},
\[
\sigma(\pi)\geq 1+\frac{1}{\lambda},
\]
and hence $0<\sigma(\pi)^{-1}<1$. Thus $\sigma(\vy)$ is a convex
combination of $\sigma(\vs)$ and $\sigma(\vx)$. Since
$\sigma(\vs),\sigma(\vx)\in\cW$, the convexity of $\cW$ implies that
$\sigma(\vy)\in\cW$.

\medskip
\noindent\textbf{Case 2.} Assume $\sigma(\pi)<0$.

By \Cref{lemma:lower_bound_on_conjugates_of_primes}, $\sigma(\pi)\le -1-\frac{1}{\lambda}$.
Set
\[
\vv=\sigma(\vx)-\sigma(\vs),
\qquad
c=\sigma(\pi)^{-1}\in\left[-\frac{\lambda}{1+\lambda},0\right).
\]
Then
\[
\sigma(\vy)-\sigma(\vs)=c\vv.
\]
Since $\sigma(\vx),\sigma(\vs)\in\cW$, we have
\[
\zero,\vv\in \cW-\sigma(\vs).
\]
By the geometric condition \eqref{eq:geometric_condition},
\[
-(\cW-\sigma(\vs))
\subset
\left(1+\frac{1}{\lambda}\right)(\cW-\sigma(\vs)).
\]
Hence
\[
-\frac{\vv}{1+\lambda^{-1}}\in \cW-\sigma(\vs).
\]
Let
\[
t=-(1+\lambda^{-1})c\in(0,1].
\]
Since $\cW-\sigma(\vs)$ is convex and contains $\zero$, we get
\[
c\vv
=
t\left(-\frac{\vv}{1+\lambda^{-1}}\right)
+(1-t)\zero
\in
\cW-\sigma(\vs).
\]
This implies that $\sigma(\vy)=c\vv+ \sigma(\vs) \in \cW,$
a contradiction to the visibility of $\vx$
from $\vs$. This proves
\begin{equation}
\label{eq:visible_from_point_s_II}
\Lambda_{\vis,\{\vs\}}
\subseteq
\left\{
\vx\in\Lambda:
\gcd(\vx-\vs)=1,\
\sigma\!\left(\frac{\vx-\vs}{\lambda}\right)\notin \cW-\sigma(\vs)
\right\}.
\end{equation}

We now prove the reverse inclusion $\supseteq$ in \eqref{eq:visible_from_point_s}.
Suppose, toward a contradiction, that
\[
\vx\in\Lambda\setminus\Lambda_{\vis,\{\vs\}},
\qquad
\gcd(\vx-\vs)=1,
\qquad
\sigma\!\left(\frac{\vx-\vs}{\lambda}\right)\notin \cW-\sigma(\vs).
\]

Then there exists a point $\vy \in \Lambda$ that is also in the interior of the straight line segment between $\vs$ and $\vx$.
Write $\vy = \vs + t(\vx - \vs)$ with $t \in (0,1)  \medcap K$. Since $\Lambda$ is locally finite, we may choose $t$ minimal such that $\vy \in \Lambda_{\vis,\{\vs\}}$.

Since $t \in K$, we may write $t = b/a$ with $a,b \in \cO_K$ coprime. If $b$ is not a unit, then
\[
\vy - \vs = \frac{b}{a}(\vx - \vs)
\]
implies that $b \mid (\vy - \vs)$, contradicting
$\gcd(\vy-\vs)=1$, which follows from the already proved inclusion
\eqref{eq:visible_from_point_s_II}.
Thus $b$ is a unit, and we may write $a = t^{-1} \in \cO_K$. 

If $a$ is not a unit, then $\gcd(\vx - \vs) \neq 1$, which is impossible.
Suppose that $a$ is a unit. Since $t = a^{-1} \in (0,1)$, it follows that $a = \lambda^{k}$ for some $k \in \mathbb{N}$.
Since $\vy = \vs + \lambda^{-k}(\vx - \vs) \in \Lambda$, we have
\[
\sigma(\vy) = \sigma(\vs) + \sigma(\lambda^{-k})\sigma(\vx - \vs) \in \cW,
\]
and hence
\begin{equation} \label{eq:point_sigma_x_minus_s_in_W}
\sigma(\lambda^{-k})\sigma(\vx - \vs) \in \cW - \sigma(\vs).
\end{equation}
As discussed above, for all the fields satisfying the Hammarhjelm condition, $\sigma(\lambda) = -1/\lambda$. We have 
\[
|\sigma(\lambda^{-k})| = (-1)^k\sigma(\lambda^{-k}) = \lambda^k \ge - \sigma(\lambda^{-1}) =\lambda >1.
\]
Recall that $\cW - \sigma(\vs)$ is convex and contains $\zero$. We distinguish two cases according to the parity of $k$.

\medskip
\noindent \textbf{Case 1.} $k$ is odd.

If $k=1$ then \eqref{eq:point_sigma_x_minus_s_in_W} implies 
\[
\sigma(\lambda^{-1})\sigma(\vx - \vs) \in \cW - \sigma(\vs).
\]
If $k\ge 3$, letting $\ell=  \sigma(\lambda^{k-1}) \in (0,1]$, \eqref{eq:point_sigma_x_minus_s_in_W} implies that
\[
\sigma(\lambda^{-1})\sigma(\vx - \vs) = \ell\sigma(\lambda^{-k})\sigma(\vx - \vs) \in \cW - \sigma(\vs).
\]

\medskip
\noindent \textbf{Case 2.} $k$ is even.

Combining this with \eqref{eq:point_sigma_x_minus_s_in_W}, we obtain
\[
\sigma(\lambda^{-1})^2\sigma(\vx-\vs)
=
-\lambda\sigma(\lambda^{-1})\sigma(\vx-\vs)
\in \cW-\sigma(\vs).
\]
Equivalently,
\[
\lambda\sigma(\lambda^{-1})\sigma(\vx-\vs)
\in -(\cW-\sigma(\vs)).
\]
By the hypothesis
\[
-(\cW-\sigma(\vs))
\subset
\left(1+\frac{1}{\lambda}\right)(\cW-\sigma(\vs)),
\]
we get
\[
\lambda\sigma(\lambda^{-1})\sigma(\vx-\vs)
\in
\left(1+\frac{1}{\lambda}\right)(\cW-\sigma(\vs)).
\]
By \Cref{lemma:inequality_for_fundamental_unit}, 
$1+\frac{1}{\lambda}\leq \lambda$.
Since $\cW-\sigma(\vs)$ is convex and contains $\zero$, the dilates are
nested under positive scaling. Hence
\[
\lambda\sigma(\lambda^{-1})\sigma(\vx-\vs)
\in
\lambda(\cW-\sigma(\vs)),
\]
and therefore
\[
\sigma(\lambda^{-1})\sigma(\vx-\vs)\in \cW-\sigma(\vs).
\]
In both cases,
\[
\sigma\!\left(\frac{\vx-\vs}{\lambda}\right)\in \cW-\sigma(\vs),
\]
contradicting our assumption. This proves the reverse inclusion and completes
the proof.
\end{proof}

\section{Proof of  \texorpdfstring{\Cref{thm:visibility_from_set_S}}{}}

In order to prove \Cref{thm:visibility_from_set_S}
we will establish a series of auxiliary lemmas. 
In \Cref{sec:proof_of_theorem_density_of_simultaneously_visible_points} we complete the proof.

\subsection{Window exclusions}

\begin{lemma}
\label{lemma:W_s_is_contained_in_W}
Let $\cW\subseteq\R^d$ be a bounded, convex, and centrally symmetric set with nonempty interior, let
$\vs\in\Lambda(\cW,\cL)$ satisfy the geometric condition \eqref{eq:geometric_condition}.
Define
\[
\cW_{\vs}
\df
\frac{1}{\lambda}\bigl(\cW+(\lambda+1)\sigma(\vs)\bigr).
\]
Then
\[
\cW_{\vs}\subseteq\cW.
\]
\end{lemma}

\begin{proof}
Since $\vs\in\Lambda(\cW,\cL)$, we have $\sigma(\vs)\in\cW$. Hence
\[
\zero\in \cW-\sigma(\vs).
\]
Since $\cW$ is convex, the translated set $\cW-\sigma(\vs)$ is also convex.
Therefore, because it contains $\zero$, we have
\[
t(\cW-\sigma(\vs))\subseteq \cW-\sigma(\vs)
\text{ for every }
t\in[0,1].
\]

Let $\vy\in\cW_{\vs}$. By definition, there exists $\vw\in\cW$ such that
\[
\vy
=
\frac{\vw+(\lambda+1)\sigma(\vs)}{\lambda}
=
\sigma(\vs)+\frac{\vw+\sigma(\vs)}{\lambda}.
\]
Thus
\[
\vy-\sigma(\vs)
=
\frac{\vw+\sigma(\vs)}{\lambda}.
\]
Since $\cW$ is centrally symmetric, $-\vw\in\cW$. Hence
\[
\vy-\sigma(\vs)
=
-\frac{1}{\lambda}\bigl((-\vw)-\sigma(\vs)\bigr)
\in
-\frac{1}{\lambda}(\cW-\sigma(\vs)).
\]
By hypothesis,
\[
-(\cW-\sigma(\vs))
\subseteq
\left(1+\frac{1}{\lambda}\right)(\cW-\sigma(\vs)).
\]
Multiplying by $1/\lambda$, we obtain
\[
-\frac{1}{\lambda}(\cW-\sigma(\vs))
\subseteq
\frac{1+\lambda^{-1}}{\lambda}(\cW-\sigma(\vs)).
\]
Using Lemma \ref{lemma:inequality_for_fundamental_unit}, we get
\[
0<\frac{1+\lambda^{-1}}{\lambda}\leq 1.
\]
Since $\cW-\sigma(\vs)$ is convex and contains $\zero$, it follows that
\[
\frac{1+\lambda^{-1}}{\lambda}(\cW-\sigma(\vs))
\subseteq
\cW-\sigma(\vs).
\]
Therefore
\[
\vy-\sigma(\vs)\in\cW-\sigma(\vs),
\]
and hence $\vy\in\cW$. Since $\vy\in\cW_{\vs}$ was arbitrary, we conclude
that
\[
\cW_{\vs}\subseteq\cW.
\]
\end{proof}

We denote the set of $\vs$-primitive points in
$\Lambda(\cW_{\vs},\cL)$ by
\[
\Lambda_{\pr,\{\vs\}}\!(\cW_{\vs},\cL)
\df
\left\{
\vx\in\Lambda(\cW_{\vs},\cL):
\gcd(\vx-\vs)=1
\right\}.
\]
More generally, if $\cW'\subset\cW$, we denote the set of
$S$-primitive points in $\Lambda(\cW',\cL)$ by
\[
\Lambda_{\pr,S}(\cW',\cL)
\df
\left\{
\vx\in\Lambda(\cW',\cL):
\gcd(\vx-\vs)=1 \text{ for all } \vs\in S
\right\}.
\]

By \Cref{lemma:visibility_from_a_point_s,lemma:W_s_is_contained_in_W}, we have
\begin{equation} \label{eq:visible_points_in_terms_of_primitive_points}
\begin{split}
\Lambda_{\vis,S}
&=
\Lambda_{\pr,S}(\cW,\cL)
\setminus
\bigcup_{\vs\in S}
\left(
\Lambda_{\pr,\{\vs\}}\!(\cW_{\vs},\cL)
\medcap
\Lambda_{\pr,S}\!(\cW,\cL)
\right). \\
\end{split}
\end{equation}
Applying the inclusion-exclusion principle to the union in
\eqref{eq:visible_points_in_terms_of_primitive_points}, we obtain 
\begin{equation}
\label{eq:inclusion_exclusion_for_visible_points_expanded}
\begin{aligned}
\card{\Lambda_{\vis,S}\medcap TD}&
=
\card{\Lambda_{\pr,S}
(\cW,\cL)  \medcap TD} \\
- & 
\sum_{\emptyset\neq I\subseteq S}
(-1)^{\card{I}+1}
\card{
\Lambda_{\pr,S}(\cW,\cL)
\medcap
\bigcap_{\vs\in I}\Lambda_{\pr,\{\vs\}}\!(\cW_{\vs},\cL)
\medcap
TD
}.
\end{aligned}
\end{equation}

\begin{lemma}
\label{lemma:intersection_of_S_primitive_and_singleton_primitive}
Let $\cW\subseteq\R^d$ be a bounded, convex, and centrally symmetric set with nonempty interior, let
$\vs\in\Lambda(\cW,\cL)$ satisfying the condition \eqref{eq:geometric_condition}.
Let $I\subseteq S$ be nonempty. Then
\[
\Lambda_{\pr,S}(\cW,\cL)
\medcap
\bigcap_{\vs\in I}
\Lambda_{\pr,\{\vs\}}\!(\cW_{\vs},\cL)
=
\Lambda_{\pr,S}(\cW_I,\cL).
\]
\end{lemma}

\begin{proof}
Let $\vx\in
\Lambda_{\pr,S}(\cW,\cL)
\medcap
\bigcap_{\vs\in I}\Lambda_{\pr,\{\vs\}}\!(\cW_{\vs},\cL)$.
Since $\vx\in\Lambda_{\pr,S}(\cW,\cL)$, we have
\[
\gcd(\vx-\vs)=1
\qquad
\text{for all }\vs\in S.
\]
Moreover, for every $\vs\in I$, the condition
$\vx\in\Lambda_{\pr,\{\vs\}}\!(\cW_{\vs},\cL)$ implies
$\vx\in\Lambda(\cW_{\vs},\cL)$. Hence
\[
\sigma(\vx)\in\bigcap_{\vs\in I}\cW_{\vs}=\cW_I,
\]
and therefore $\vx\in\Lambda(\cW_I,\cL)$. Combining this with the
coprimality conditions for all $\vs\in S$, we obtain
\[
\vx\in\Lambda_{\pr,S}(\cW_I,\cL).
\]
Thus
\[
\Lambda_{\pr,S}(\cW,\cL)
\medcap
\bigcap_{\vs\in I}\Lambda_{\pr,\{\vs\}}\!(\cW_{\vs},\cL)
\subseteq
\Lambda_{\pr,S}(\cW_I,\cL).
\]

Conversely, let $\vx\in\Lambda_{\pr,S}(\cW_I,\cL)$. Then
\[
\sigma(\vx)\in\cW_I=\bigcap_{\vs\in I}\cW_{\vs}
\]
and
\[
\gcd(\vx-\vs)=1
\qquad
\text{for all }\vs\in S.
\]
Since $\cW_I\subseteq\cW_{\vs}$ for every $\vs\in I$, it follows that
\[
\vx\in\Lambda_{\pr,\{\vs\}}\!(\cW_{\vs},\cL)
\qquad
\text{for every }\vs\in I.
\]
Also, by \Cref{lemma:W_s_is_contained_in_W}, we have
$\cW_{\vs}\subseteq\cW$ for every $\vs\in I$, and hence
\[
\cW_I\subseteq\cW.
\]
Therefore $\vx\in\Lambda(\cW,\cL)$. Together with the coprimality conditions
for all $\vs\in S$, this gives
\[
\vx\in\Lambda_{\pr,S}(\cW,\cL).
\]
Thus
\[
\Lambda_{\pr,S}(\cW_I,\cL)
\subseteq
\Lambda_{\pr,S}(\cW,\cL)
\medcap
\bigcap_{\vs\in I}\Lambda_{\pr,\{\vs\}}\!(\cW_{\vs},\cL),
\]
and the proof is complete.
\end{proof}

Using \eqref{eq:inclusion_exclusion_for_visible_points_expanded} and \Cref{lemma:intersection_of_S_primitive_and_singleton_primitive}, we obtain the following \namecref{corollary:set_visible_points_in_terms_of_primitive_points_2}.

\begin{corollary} \label{corollary:set_visible_points_in_terms_of_primitive_points_2}
Under the hypotheses of \Cref{thm:visibility_from_set_S}, 
for every sufficiently large $T>0$,
\begin{equation*}
%\label{eq:simplified_inclusion_exclusion_expression}
\begin{split}
\card{\Lambda_{\vis,S}(\cW, \cL)  \medcap TD} 
 = &
\card{\Lambda_{\pr,S}(\cW, \cL)  \medcap TD}\\
&
\ - \sum_{\emptyset \neq I \subseteq S} (-1)^{\card{I}+1} \card{\Lambda_{\pr,S}(\cW_I, \cL)  \medcap TD}.
\end{split}    
\end{equation*}    
\end{corollary}

The final step in the proof of \Cref{thm:visibility_from_set_S} is the following
\namecref{lemma:density_of_primitive_points}.

\begin{lemma} \label{lemma:density_of_primitive_points}
Under the hypotheses of \Cref{thm:visibility_from_set_S}, let either
$\cW'=\cW$, or $\cW'=\cW_I$ for some nonempty subset
$I\subseteq S$. Then the density of $\Lambda_{\pr,S}(\cW',\cL)$ along $\{TD\}_{T>0}$ exists and is given by:
\[
\theta\bigl(\Lambda_{\pr,S}(\cW',\cL)\bigr)
=
\frac{\vol(\cW')}{\covol(\cL)}
\prod_{\fp\in\cP_{\cO_K}}
\left(1-\frac{s(\fp)}{N(\fp)^d}\right).
\]
\end{lemma}

\subsection{Completion of the proof of \texorpdfstring{\Cref{thm:visibility_from_set_S}}{}}
\label{sec:proof_of_theorem_density_of_simultaneously_visible_points}

We assume \Cref{lemma:density_of_primitive_points} for now and postpone its proof to the next subsection. Combining 
\Cref{corollary:set_visible_points_in_terms_of_primitive_points_2} with \Cref{lemma:density_of_primitive_points},  we obtain that the density of
$\Lambda_{\vis,S}(\cW,\cL)$ along the sequence $\{TD\}_{T>0}$ exists. Moreover,
\[
\begin{aligned}
\theta(\Lambda_{\vis,S}(\cW,\cL))
&=
\lim_{T\to\infty}
\frac{\card{\Lambda_{\pr,S}(\cW,\cL)\medcap TD}}
{\vol(TD)} \\
&\quad
-
\sum_{\emptyset\neq I\subseteq S}
(-1)^{\card{I}+1}
\lim_{T\to\infty}
\frac{\card{\Lambda_{\pr,S}(\cW_I,\cL)\medcap TD}}
{\vol(TD))} \\
&=
\frac{\vol(\cW)}{\covol(\cL)}
\prod_{\fp\in\cP_{\cO_K}}
\left(1-\frac{s(\fp)}{N(\fp)^d}\right) \\
&\quad
-
\sum_{\emptyset\neq I\subseteq S}
(-1)^{\card{I}+1}
\frac{\vol(\cW_I)}{\covol(\cL)}
\prod_{\fp\in\cP_{\cO_K}}
\left(1-\frac{s(\fp)}{N(\fp)^d}\right) \\
&=
\dens(\cL)
\left(
\vol(\cW)
-
\sum_{\emptyset\neq I\subseteq S}
(-1)^{\card{I}+1}\vol(\cW_I)
\right)
\prod_{\fp\in\cP_{\cO_K}}
\left(1-\frac{s(\fp)}{N(\fp)^d}\right).
\end{aligned}
\]
This is the formula claimed in \Cref{thm:visibility_from_set_S}. 

It remains to prove \Cref{lemma:density_of_primitive_points}.

\subsection{Proof of \texorpdfstring{\Cref{lemma:density_of_primitive_points}}{}}

For a rational prime $p$, let $M_p$ denote the set of prime ideals of
$\cO_K$ lying above $p$. Since $K$ is quadratic, we have
\[
\card{M_p}\leq 2.
\]
Let $p_i$ denote the $i$th rational prime, and define
\[
\EE_t\df \bigcup_{i=1}^t M_{p_i},
\qquad t\in\N.
\]
For a prime ideal $\fp$, define
\[
\fp^d \df \{(x_1,\ldots, x_d)\in \cO_K^d : x_i\in \fp\}.
\]
For $\cW'$ as in \Cref{lemma:density_of_primitive_points}, define
\[
\Lambda_{\pr,S, \EE_t}(\cW',\cL)
\df
\left\{
\vx\in\Lambda(\cW',\cL):
\vx-\vs \notin \fp^d
\text{ for all } \fp\in\EE_t \text{ and all } \vs\in S
\right\}.
\]
Equivalently, $\Lambda_{\pr,S,\EE_t}(\cW',\cL)$ is obtained from
$\Lambda(\cW',\cL)$ by imposing only the primitivity conditions coming from
prime ideals above the first $t$ rational primes.

In the following \namecref{lemma:density_of_finite_prime_obstructions}, we show that the density of $\Lambda_{\pr,S,\EE_t}(\cW',\cL)$ along averaging sets $TD$ exists.

\begin{lemma} \label{lemma:density_of_finite_prime_obstructions}
Under the hypotheses of \Cref{lemma:density_of_primitive_points}, the density
of $\Lambda_{\pr,S,\EE_t}(\cW',\cL)$ along the sequence $\{TD\}_{T>0}$ exists.
Moreover,
\begin{equation}
\label{eq:density_result_for_finite-primes}
\theta(\Lambda_{\pr, S,\EE_t})= \frac{\vol(\cW')}{\covol(\cL)}\cdot \prod_{\fp\in \EE_t}\left(1-\frac{s(\fp)}{N(\fp)^d}\right).
\end{equation}
\end{lemma}

\begin{proof}
Let $G=\R^d$ be the physical space, and let the internal space be given by
\[
H= \R^d\times \prod_{\fp\in \mathcal{P}_{\mathcal{O}_K}}H_{\fp}, 
\]
where $H_{\fp}= \left(\quotient{\cO_K}{\fp}\right)^d$.
In the group  $H$, each finite factor is equipped with the normalized counting measure.
Both $G$ and $H$ are $\sigma$-compact and LCAGs, and the group $H$ is compactly generated. Let $\cL''$ be the image of the diagonal embedding
\[
\cO_K^d \hookrightarrow G\times H,
\qquad
\vx \hookrightarrow (\vx, \sigma(\vx) , (\pi_{\fp}(\vx))_{\fp\in \mathcal{P}_{\mathcal{O}_K}}).
\]
The restriction $\pi_{\phys}''|_{\cL''}: \cL''\to G$ is injective. 
Using the Chinese remainder theorem and the fact that $\pi_{{\internal}}(\cL)$ is dense in $\mathbb{R}^d$ it can be shown that the image $\pi_{\internal}''(\cL'')$ is dense in $H$. 
The projection 
\[
G \times H \longrightarrow \R^{2d}, \qquad  (\vu, \vv , (\vw_{\fp})_{\fp\in \mathcal{P}_{\mathcal{O}_K}}) \longmapsto (\vu, \vv)
\]
descends to a well-defined continuous surjective map
\[
\quotient{(G \times H)}{\cL''} \longrightarrow \quotient{\R^{2d}}{\cL}
\]
whose kernel is $\prod_{\fp \in \mathcal{P}_{\mathcal{O}_K}}H_\fp$.
Since $\cL$ is a cocompact lattice, $\quotient{G \times H}{\cL''}$ is compact, which shows that $\cL''$ is a cocompact lattice.

Define
\[
\cW'' \ \df \ \cW'\times \prod_{\fp\in \EE_t}\cW_{\fp}\times \prod_{\substack{\fp\in\cP_{\cO_K}\\ \fp\notin \EE_t}}H_{\fp},
\]
where 
\[
\cW_{\fp}
\ \df \ 
\left(\quotient{\cO_K}{\fp}\right)^d{\big\backslash}\pi_{\fp}(S).
\]
By construction, the associated cut-and-project set $\Lambda(\cW'', \cL'')$
consists exactly of
those $\vx\in\cO_K^d$ such that
\[
\sigma(\vx)\in\cW'
\qquad\text{and}\qquad
\pi_{\fp}(\vx)\notin\pi_{\fp}(S)
\quad
\text{for every }\fp\in\EE_t.
\]
Equivalently,
\[
\Lambda(\cW'',\cL'')
=
\Lambda_{\pr,S,\EE_t}(\cW',\cL).
\]

Suppose that for a prime ideal $\fp \in \EE_t$, we have $s(\fp)=N(\fp)^d$. In this case $\cW''= \varnothing$. Therefore, $\Lambda_{\pr,S,\EE_t}(\cW',\cL)=\varnothing$.  Since $\Lambda_{\pr,S}(\cW',\cL) \subseteq\Lambda_{\pr,S,\EE_t}(\cW',\cL)$, it follows that \eqref{eq:density_result_for_finite-primes} holds true.

Now suppose that $s(\fp)<N(\fp)^d$ for all $\fp \in \EE_t$. In this case, the set $\cW''$ is a regular window: it is relatively compact,  has nonempty interior, and its boundary has Haar measure zero.
Furthermore,
\[
\vol(\cW'')
=
\vol(\cW')
\prod_{\fp\in\EE_t}
\left(1-\frac{s(\fp)}{N(\fp)^d}\right).
\]
It follows from the construction that $\Lambda(\cW'',\cL'')$ is an irreducible cut-and-project set. Applying the density formula \eqref{eq:density_irreducible_CPS} for irreducible cut-and-project sets, we get
\[
\begin{split}
\theta(\Lambda_{\pr,S, \EE_t})
%& 
= \theta(\Lambda(\cW'',\cL''))
%\\ &
&= \frac{\vol(\cW')}{\covol(\cL'')} \prod_{\fp\in \EE_t}\left(1- \frac{s(\fp)}{N(\fp)^d}\right)
\\
&
= \frac{\vol(\cW')}{\covol(\cL)} \prod_{\fp\in \EE_t}\left(1- \frac{s(\fp)}{N(\fp)^d}\right)
\end{split}  
\]
completing the proof of \Cref{lemma:density_of_finite_prime_obstructions}.
\end{proof}

Since
\[
\Lambda_{\pr,S}(\cW',\cL)
\subseteq
\Lambda_{\pr, S,\EE_t}(\cW',\cL)
\]
for every $t\in\N$, \Cref{lemma:density_of_finite_prime_obstructions} gives
\[
\overline{\theta}\bigl(\Lambda_{\pr,S}(\cW',\cL)\bigr)
\leq
\theta\bigl(\Lambda_{\pr,S,\EE_t}(\cW',\cL)\bigr)
=
\frac{\vol(\cW')}{\covol(\cL)}
\prod_{\fp\in\EE_t}
\left(1-\frac{s(\fp)}{N(\fp)^d}\right).
\]
Letting $t\to\infty$, we obtain
\[
\overline{\theta}\bigl(\Lambda_{\pr,S}(\cW',\cL)\bigr)
\leq
\frac{\vol(\cW')}{\covol(\cL)}
\prod_{\fp\in\cP_{\cO_K}}
\left(1-\frac{s(\fp)}{N(\fp)^d}\right).
\]

We will make use of the following statement in the sequel.
\begin{lemma} \label{lemma:aux_identity}
Suppose that $B \subseteq A \subset \R^d$ are nonempty sets and $\theta(A)$ exists along an averaging sequence $\{D_T\}_{T >0}$.
Then,
\[
\underline{\theta}\bigl(B)
=
\theta\bigl(A)
-
\overline{\theta}\left(
A
\setminus
B
\right).
\]
\end{lemma}
Observe that for   $B=\Lambda_{\pr,S}(\cW',\cL) $ and $A= \Lambda_{\pr,S,\EE_t}(\cW',\cL)$, we have $B \subseteq A$ and $\theta(A)$ exists by  \Cref{lemma:density_of_finite_prime_obstructions} along the averaging sequence $\{TD\}_{T>0}$. Then \Cref{lemma:aux_identity} implies the following.
\begin{corollary} \label{lemma:auxiliary_identity}
For every $t\in\N$,
\[
\underline{\theta}\bigl(\Lambda_{\pr,S}(\cW',\cL)\bigr)
=
\theta\bigl(\Lambda_{\pr,S,\EE_t}(\cW',\cL)\bigr)
-
\overline{\theta}\left(
\Lambda_{\pr,S,\EE_t}(\cW',\cL)
\setminus
\Lambda_{\pr,S}(\cW',\cL)
\right).
\]
\end{corollary}

\begin{proof}[Proof of \Cref{lemma:aux_identity}]
Because $B \subseteq A$, we can decompose $A$ into the disjoint union
\[
A = B \sqcup (A \setminus B).
\]
For $T>0$, we have
\[
\frac{
\card{B  \medcap D_T}
}{\vol(D_T)} = 
\frac{
\card{A  \medcap D_T}
}{\vol(D_T)} - 
\frac{
\card{(A \setminus B)  \medcap D_T}
}{\vol(D_T)}.
\]
Let
\[
a_T\df \frac{
\card{A  \medcap D_T}
}{\vol(D_T)}
\quad
\text{and}
\quad
c_T \df \frac{
\card{(A \setminus B)  \medcap D_T}
}{\vol(D_T)}.
\]
Let $T_n$ be a sequence such that $T_n\xrightarrow[]{n\to \infty} \infty$ and
\[
\lim_{n\to\infty}c_{T_n} = \limsup_{T \to \infty}c_{T}.
\]
Hence $\displaystyle\lim_{n\to \infty}a_{T_n} = \theta(A)$.
On one hand
\[
\theta(A)-\limsup_{T\to\infty}c_T
= 
\lim_{n\to\infty}a_{T_n} 
-
\lim_{n\to\infty}c_{T_n}
=
\lim_{n\to\infty}(a_{T_n}-c_{T_n}) 
\ge 
\liminf_{T\to\infty}(a_{T}-c_{T}).
\]
On the other hand
\[
\liminf_{T\to\infty}(a_{T}-c_{T}) \ge 
\liminf_{T\to\infty}a_{T} + \liminf_{T\to\infty}(-c_T) = \theta(A)-\limsup_{T\to\infty}c_T.
\]
Combining the previous inequalities we obtain $\underline{\theta}(B) = \theta(A) - \overline{\theta}(A \setminus B).$
\end{proof}

We have the following intermediate result.

\begin{lemma} \label{lemma:tail_prime_obstructions_have_zero_density}
Under the hypotheses of \Cref{lemma:density_of_primitive_points},
\[
\lim_{t\to\infty}
\overline{\theta}\left(
\Lambda_{\pr, S,\EE_t}(\cW',\cL)
\setminus
\Lambda_{\pr,S}(\cW',\cL)
\right)
=
0.
\]
\end{lemma}

The proof of \Cref{lemma:density_of_primitive_points} follows from \Cref{lemma:tail_prime_obstructions_have_zero_density} in combination with Lemma \ref{lemma:density_of_finite_prime_obstructions} and \Cref{lemma:auxiliary_identity}.

\begin{proof}[Proof of \Cref{lemma:tail_prime_obstructions_have_zero_density}]
Define
\[
A_t
\df
\Lambda_{\pr, S,\EE_t}(\cW',\cL)
\setminus
\Lambda_{\pr,S}(\cW',\cL).
\]
Thus, $A_t$ consists of those points that satisfy the primitivity conditions for every $\vs\in S$ and every prime ideal $\fp\in \EE_t$. However, they fail the primitivity condition for at least one pair $(s,\mathfrak p)$, where $\vs\in S$ and $\mathfrak p\notin \EE_t$.

For $g\in\cO_K$, let $a_g$ be the diagonal matrix
\[
a_g
\df
\operatorname{diag}(
\underbrace{g,\ldots,g}_{\text{$d$ times}},
\underbrace{\sigma(g),\ldots,\sigma(g)}_{\text{$d$ times}}
),
\]
and define
\begin{equation*} %\label{eq:lattice_L_g}
\cL_{(g)}
\df
a_g\cL
=
\{(g\vx,\sigma(g\vx)):\vx\in\cO_K^d\}.
\end{equation*}
Since $\cO_K$ is a ring, $\cL_{(g)}\subseteq\cL$, and
\begin{equation*} %\label{eq:covolume_g}
\covol(\cL_{(g)})
=
|N(g)|^d\covol(\cL).
\end{equation*}
As $\cO_K$ is a principal ideal domain, if $\fp=(g)$ we write
$\cL_{\fp}$ for $\cL_{(g)}$.
We claim that
\begin{equation} \label{eq:tail_obstruction_cover}
A_t
\subseteq
\bigcup_{\vs\in S}
\bigcup_{\substack{\fp\in\cP_{\cO_K}\\ \fp\notin \EE_t}}
\left(
\vs+\Lambda(\cW'-\sigma(\vs),\cL_{\fp})
\right).
\end{equation}
Indeed, let $\vx\in A_t$ be an arbitrary point. Then
there exist $\vs_0\in S$ and a prime ideal $\fp\notin \EE_t$ such that
\[
 \fp 
\Bigm|
(\vx-\vs_0)
,
\]
or equivalently $\vx-\vs_0\in\fp^d$. Since
$\sigma(\vx)\in\cW'$, we have
\[
\sigma(\vx-\vs_0)\in \cW'-\sigma(\vs_0).
\]
Thus
\[
\vx\in
\vs_0+\Lambda(\cW'-\sigma(\vs_0),\cL_{\fp}),
\]
which proves \eqref{eq:tail_obstruction_cover}.
\begin{claim}{}
\label{claim:unions_of_sets}There exists a constant $c>0$, such that for every $t\in \mathbb{N}$ and every  $T > p_t^2$,
\begin{equation} \label{eq:tail_obstruction_cover_intersect_with_TD}
A_t\medcap TD
\subseteq
\bigcup_{\vs\in S}
\bigcup_{\substack{\fp\in\cP_{\cO_K}\\ \fp\notin \EE_t \\
N(\fp) \le cT}}
\left(
\vs+\Lambda(\cW'-\sigma(\vs),\cL_{\fp})\medcap TD
\right).
\end{equation}
That is, the union on the right-hand side of \eqref{eq:tail_obstruction_cover_intersect_with_TD} is taken over finitely many prime ideals $\fp \notin \EE_t$ with norm at most $cT$.
\end{claim}

\begin{proof}[Proof of \cref{claim:unions_of_sets}]
\renewcommand{\qedsymbol}{$ \blacktriangle$}
Fix $\vs \in S$, and suppose 
$$\vx\in (\vs+\Lambda(\cW'-\sigma(\vs),\cL_{\fp}))\medcap TD,\qquad \vx\neq \vs.$$
Then $\vx - \vs \in \fp^d$. Choose a coordinate $j$ such that $x_j - s_j \ne 0$. Because $x_j - s_j \in \fp$, we have $(x_j - s_j) \subseteq \fp$, where $(x_i-s_j)$ is the ideal generated by $x_j-s_j$.
For ideals, inclusion reverses the norm inequality; hence$$N(\fp) \le N((x_j - s_j)) = \left\vert{}N(x_j - s_j)\right\vert{}.$$

Now $N(x_j - s_j) = (x_j - s_j)\sigma(x_j - s_j)$. Since $D$ is bounded and $\vx \in TD$, there is a constant $C_D > 0$ depending only on $D$, such that $\vert{}x_j\vert{} \le C_D T$. Since $S$ is finite, we obtain
\begin{equation}\label{eq.bound_of_xj_sj}
\vert{}x_j - s_j\vert{} \le c_1 T,\qquad \forall\,  1\leq j \leq d,\, \forall\, \vs\in S, 
\end{equation}
where $c_1$ is a constant independent of $\fp,\vs,T$.
Also, $\sigma(x_j - s_j) = \sigma(x_j) - \sigma(s_j)$. The window $\cW'$ is bounded. 
Thus, given $\vy= (y_1,\dots,y_d) \in \cW'-\sigma(\vs)$, there exists $c_{\vs}>0$ such that 
\[
|y_j| \le c_{\vs}, 
\quad \qquad
1\le j \le d.
\]
The constant $c_2 \df \max_{\vs \in S}\{c_{\vs}\}>0$ is  independent of $\vx,\vs, T$ and $\fp$, and satisfies
\begin{equation}
\label{eq.bound_sigma_of_xj_sj}
\vert{}\sigma(x_j - s_j)\vert{} \le c_2.
\end{equation}
Combining \eqref{eq.bound_of_xj_sj} and \eqref{eq.bound_sigma_of_xj_sj}, we obtain
$$N(\fp)\leq c_1c_2T.$$
Let $H_n$ be the number of ideals in $\mathcal{O}_K$ of norm $n$. It
follows from \cite[Theorem 39]{marcus_number_2018} that 
\begin{equation*}
%\label{equation for number of ideals} 
H_n\leq H \cdot n^{1/2},  
\end{equation*}
where $H>0$ is a constant independent of $n$.
Hence, only finitely many prime ideals can occur in \eqref{eq:tail_obstruction_cover_intersect_with_TD}. This completes the proof of the \namecref{claim:unions_of_sets}.
\end{proof}

For $\vs \in S$ and $\fp \in \cP_{\cO_K}$, write
\[
N_{\vs,\fp}(T) \df \card{ \left( \vs + \Lambda(\cW' - \sigma(\vs), \cL_{\fp}) \right) \medcap TD }.
\]
By \eqref{eq:tail_obstruction_cover_intersect_with_TD}, we have

\begin{equation}
\label{eq.A(T)cap_TD}
\card{A_t \medcap TD} \le \sum_{s \in S} \sum_{\fp \notin \EE_t} N_{s,\fp}(T),
\end{equation}
where the second sum runs over finitely many primes.
By \cite[Lemma 5.4]{gringlaz_error_2025} we have for $N(\fp)\leq T$:

\begin{equation}
\label{eq:bound_of_N(s,p,T)}
\begin{split}
N_{\vs,\fp}(T) &\le \frac{\vol(TD)\cdot \vol(\cW')}{\covol(\cL_{\fp})} + C_1 \left( \frac{T}{N(\fp)} \right)^{d-\frac{1}{2}}\\
N_{\vs,\fp}(T) &\le C_0 \frac{T^d}{N(\fp)^d} + C_1 \left( \frac{T}{N(\fp)} \right)^{d-\frac{1}{2}}\\
\frac{N_{\vs,\fp}(T)}{\vol(TD)}&\leq \frac{C_2}{N(\fp)^d} + \frac{C_3}{T^{1/2}N(\fp)^{d-\frac{1}{2}}},
\end{split} 
\end{equation}
where $C_0,C_1, C_2,C_3>0$ are constants independent of $\vs,\fp$ and $T$.
It follows from \cite[Lemma 4.6]{hammarhjelm_density_2022} (see, also \cite[Lemma 5.5]{gringlaz_error_2025}), that for prime ideals satisfying $N(\fp)>T$ 

\begin{equation}
\label{eq.When_N(P)_large}
 \frac{N_{\vs,\fp}(T)}{\vol(TD)} \le \frac{C_4}{N(\fp)^d}, 
\end{equation}
for some constant $C_4$ independent of $\vs, \fp$  and $T$.
Splitting $N_{\vs,\fp}(T)$ over sums of primes according to $N(\fp)\leq T$ or $N(\fp)>T$, we obtain
\[
N_{\vs,\fp}(T)
=
\sum_{\substack{\fp \notin \EE_t \\ N(\fp) \le T}} N_{s,\fp}(T) 
\quad + 
\sum_{\substack{\fp \notin \EE_t \\ T<N(\fp) \leq  c_1c_2T}} N_{s,\fp}(T).
\]
By \eqref{eq.A(T)cap_TD}, \eqref{eq:bound_of_N(s,p,T)} and \eqref{eq.When_N(P)_large}, there is a constant $C$ independent of $\vs$, $\fp$, and $T$, such that
$$\frac{\card{A_t \medcap TD}}{\operatorname{vol}(TD)} \le C\cdot \card{S} \sum_{\substack{\fp \notin \EE_t \\ N(\fp) \le T}} \frac{1}{N(\fp)^d} + \frac{C\cdot \card{S}}{T^{1/2}} \sum_{\substack{\fp \notin \EE_t \\ N(\fp) \le T}} \frac{1}{N(\fp)^{d-\frac{1}{2}}} + C\cdot \card{S} \sum_{\substack{\fp \notin \EE_t \\ N(\fp) > T}} \frac{1}{N(\fp)^d}. $$
Combining the first and third sums yields:
\begin{equation}
\label{eq:A(T)capTD_devided_by_vol(TD)}
\frac{\card{A_t \medcap TD}}{\vol(TD)} \le C\cdot \card{S} \sum_{\substack{\fp \in \cP_{\mathcal{O}_K} \\\fp \notin \EE_t}} \frac{1}{N(\fp)^d} + \frac{C\cdot \card{S}}{T^{1/2}} \sum_{\fp \in \cP_{\mathcal{O}_K}} \frac{1}{N(\fp)^{d-\frac{1}{2}}}.    
\end{equation}
Because $d \ge 2$, we have $d - \frac{1}{2} > 1$, and  hence, 
$$\sum_{\fp \in \cP_{\mathcal{O}_K}} \frac{1}{N(\fp)^{d-\frac{1}{2}}} < \infty.$$
Therefore, for fixed $t$ and taking $\limsup_{T\to\infty}$ the second term in \eqref{eq:A(T)capTD_devided_by_vol(TD)} vanishes. Thus, the upper density satisfies:
\begin{equation}\label{A(t)_upper_bound}
\overline{\theta}(A_t) \le C\cdot \card{S}\sum_{\substack{\fp \in \cP_{\mathcal{O}_K} \\\fp \notin \EE_t}} \frac{1}{N(\fp)^d}.    
\end{equation}
Finally, since $d \ge 2$, $\sum_{\fp \in \cP_{\mathcal{O}_K}} \frac{1}{N(\fp)^d} < \infty$. The sets $\EE_t$ exhaust $\cP_{\mathcal{O}_K}$ as $t \to \infty$,  
\begin{equation}
\label{eq:dedekind_zeta_conve}
\sum_{\substack{\fp \in \cP_{\mathcal{O}_K} \\\fp \notin \EE_t}} \frac{1}{N(\fp)^d}
\xrightarrow[t \to \infty]{} 0.    
\end{equation}
Combining \eqref{A(t)_upper_bound} and \eqref{eq:dedekind_zeta_conve}, we conclude that
$$\lim_{t\to \infty}\overline{\theta}(A_t)=0$$
proving \Cref{lemma:tail_prime_obstructions_have_zero_density}.
\end{proof}

\section{Proof of \texorpdfstring{\Cref{thm_weak_model}}{}}
Let $G=\R^d$ be the physical space, and let
\[
H= \R^d\times \prod_{\fp\in \mathcal{P}_{\mathcal{O}_K}}H_{\fp}
\]
be the internal space as in the proof of  \Cref{thm:visibility_from_set_S}.
Let $\cL_2$ be the image of the diagonal embedding
\[
\cO_K^d \xhookrightarrow{\quad} G\times H,
\quad
\vx \xhookrightarrow{\quad}
\left(
\vx, 
\sigma(\vx), 
\left(
\pi_{\fp}(\vx)
\right)_{\fp\in \cP_{\mathcal{O}_K}}
\right).
\]
Then $\cL_2$ is a lattice.
The projection of $\cL_2$ to $G$ is injective, while the projection to $H$ is dense. Moreover, $\covol(\cL_2)= \covol(\cL)$. 
Let 
\[
\cW_2
= 
\cW
\setminus
\left(
\bigcup_{\vs\in S}\cW_{\vs}
\right)
\times 
\prod_{\fp\in \cP_{\mathcal{O}_K}}\cW_{\fp},
\]
where
\[
\cW_{\fp}
\ \df \ 
\left(
\quotient{\cO_K}{\fp}
\right)^d
{\big\backslash}
\pi_{\fp}(S).
\]
We claim that
$$\Lambda_{\vis,S}= \Lambda(\cW_2,\cL_2).$$
In fact, by the definitions of $\Lambda_{\vis,S}$, $\Lambda_{\vis,\vs}$, and  \Cref{lemma:visibility_from_a_point_s}, we have
\begin{equation*}
\begin{split}
\vx\in \Lambda_{\vis,S} &\Longleftrightarrow \vx \in \Lambda_{\vis,\vs},\qquad \forall \vs \in S\\
& \Longleftrightarrow \sigma(\vx) \in \cW,\, \sigma(\vx)\notin \cW_{\vs},\, \gcd(\vx-\vs)=1,\, \qquad  \forall \vs \in S\\
&\Longleftrightarrow \sigma(\vx) \in \cW\setminus \bigcup_{\vs\in S}\cW_{\vs},\,\gcd(\vx-\vs)=1,\, \qquad  \forall \vs \in S\\
&\Longleftrightarrow \sigma(\vx) \in \cW\setminus \bigcup_{\vs\in S}\cW_{\vs},\, \pi_{\fp}(\vx) \notin \pi_{\fp}(S),\qquad \forall \fp \in \mathcal{P}_{\mathcal{O}_K}\\
&\Longleftrightarrow \vx \in \Lambda(\cW_2,\cL_2).
\end{split}   
\end{equation*}
Since $\cW_{\vs}\subset \cW$ and $S$ is finite, by  inclusion-exclusion
\begin{equation*}
\begin{split}
\theta_H(\cW_2)&= \vol\left(\cW\setminus \left(\bigcup_{\vs \in S}\cW_{\vs}\right)\right) \cdot \prod_{\fp \in \cP_{\mathcal{O}_K}} \left(1-\frac{s(\fp)}{N(\fp)^d}\right)\\  
&= \left(\vol(\cW)- \sum_{\varnothing \neq I\subseteq S}(-1)^{\card{I}+1}\vol\left(\medcap_{\vs\in I}\cW_{\vs}\right)\right)\cdot \prod_{\fp \in \cP_{\mathcal{O}_K}} \left(1-\frac{s(\fp)}{N(\fp)^d}\right)\\
&= \left(\vol(\cW)- \sum_{\varnothing \neq I\subseteq S}(-1)^{\card{I}+1}\vol\left(\cW_{I}\right)\right)\cdot \prod_{\fp \in \cP_{\mathcal{O}_K}} \left(1-\frac{s(\fp)}{N(\fp)^d}\right).
\end{split}    
\end{equation*}
Since $\dens(\cL)= \dens(\cL_2)$, by  \Cref{thm:visibility_from_set_S}, we get
$$\theta(\Lambda_{\vis, S})= \theta(\Lambda(\cW_2,\cL_2))= \dens(\cL_2)\theta_H(\cW_2).$$
The assumption $\theta(\Lambda_{\vis, S})>0$ implies $\theta_H(\cW_2)>0$. Since $\cW_2$ is relatively compact and $\theta_H(\cW_2)>0$, the set $\Lambda_{\vis,S}= \Lambda(\cW_2,\cL_2)$ is a weak model set of maximal density.

\section{Proof of Theorem \ref{thm:generic_density_simul_vis}} 
\label{section:random}
In this section we will discuss simultaneous visibility in $\mu$-almost every cut-and-project set. At the end we present the proof of \Cref{thm:generic_density_simul_vis}.

\subsection{Auxiliary results}
Let $n=d+m$, where $d\geq 2$ and $m\geq 1$. Let $\cL\subset \R^n$ be a lattice and let $\cW$ be a bounded convex window with nonempty interior. Assume that $\Lambda(\cW,\cL)$ is an irreducible cut-and-project set.
Since $\pi_{\phys}$ is injective, for each $\vs\in S$, we have a unique lift $\tilde{\vs}\in \cL$, such that $\pi_{\phys}(\tilde{\vs})= \vs$.
Let
\[
\tilde{S}=\{\tilde{\vs} \in \cL \,:\, \pi_{\phys}(\tilde{\vs})\in S\}.
\]
Define
\[
\cL_{\vis, \tilde{S}}\df \{\vy \in \cL \setminus \tilde{S}\, :\, \vy  \textrm{ is visible from every}\, \tilde{\vs} \in \tilde{S}\},
\]
and
\[
\Lambda\!\left(\cW,\cL_{\vis,\tilde{S}}\right)\df \{\pi_{\phys}(\vy)\, :\, \vy \in \cL_{\vis,\tilde{S}},\, \pi_{{\internal}}(\vy)\in \cW\}.
\]
\begin{lemma}
\label{lemma:projection_of_visible_points}
Let $\cL$ be a lattice in $\mathbb{R}^n$, and $\cW$ be a bounded convex window with nonempty interior.  Then
\begin{equation}
\label{eq:inclusion_lemma}
\Lambda(\cW,\cL)_{\vis, S}\subseteq \Lambda\!\left(\cW,\cL_{\vis,\tilde{S}}\right),
\end{equation}
for every finite set $S\subset \Lambda(\cW,\cL)$.
\end{lemma}

\begin{proof}
Let $\vx\in \Lambda(\cW,\cL)\setminus S$ be arbitrary. It is enough to show that if $\vx\notin \Lambda\!\left(\cW,\cL_{\vis,\tilde{S}}\right)$ then $\vx\notin \Lambda(\cW,\cL)_{\vis,S}$.  Since $\vx\in \Lambda(\cW,\cL)\setminus S$, there exists a unique $\vy \in \cL$ and $\vy\neq \tilde{\vs}$ for all $\tilde{\vs}\in \tilde{S}$, such that
\[
\vx= \pi_{\phys}(\vy) \qquad\text{and}\qquad \pi_{{\internal}}(\vy)\in \cW.
\]
The assumption $\vx\notin \Lambda\!\left(\cW,\cL_{\vis,\tilde{S}}\right)$ implies that $\vy \notin \cL_{\vis, \tilde{S}}$. Hence there exists $\tilde{\vs}_0\in \tilde{S}$, $\tilde{\vs}_0 \neq \vy$ and $t\in (0,1)$ such that
\[
\vy'= t \vy + (1-t)\tilde{\vs}_0\in \cL.
\]
Since $\pi_{\internal}(\vy), \pi_{\internal}(\tilde{\vs}_0)\in \cW$, and the window $\cW$ is convex, we have
\[
\pi_{{\internal}}(\vy')= t \pi_{\internal}(\vy)+ (1-t)\pi_{\internal}(\tilde{\vs}_0) \in \cW.
\]
Therefore
$$\pi_{\phys}(\vy')\in \Lambda(\cW,\cL),$$
and 
\[
\pi_{\phys}(\vy')= t\pi_{\phys}(\vy)+ (1-t)\pi_{\phys}(\tilde{\vs}_0)= t\vx+ (1-t)\vs_0.
\]
Since $0<t<1$, the point $\pi_{\phys}(\vy')$ lies in the open line segment $(\vs_0,\vx)$.
Hence
\[
\vx\notin \Lambda(\cW,\cL)_{\vis,S}.
\]
\end{proof}

\begin{lemma}
\label{lemma:condition_for_equality}
Let  $\cL \subset \R^{d+m}$ be  a lattice, where $d\geq 2$, $m \ge 1$, and let $\cW \subset \R^m$ be a bounded convex window with nonempty interior. 
If \eqref{eq:inclusion_lemma} is a strict inclusion,
\[
\Lambda(\cW,\cL)_{\vis, S}\subsetneq \Lambda\!\left(\cW,\cL_{\vis,\tilde{S}}\right)
\!,
\]
then there exist linearly independent vectors $\vv_1,\vv_2\in \cL$ such that
\[
\dim\left(
\mathrm{span}_\R(\vv_1,\vv_2)
\medcap 
(
\{\zero\}
\times
\mathbb{R}^m
)
\right)= 1.
\]
\end{lemma}

\begin{lemma}
\label{lemma:almost_every_simultanes_vis}
Let $n=d+m$ for positive integers $d \ge 2$ and $m \ge 1$. Suppose the window $\cW \subset \R^m$ is a bounded convex window with nonempty interior. Let $\haar$ denote the Haar measure on $\GL_n(\R)$. Then for $\haar$-a.e. $g\in \GL_n(\R)$, we have
\begin{equation*}
\Lambda(\cW, (g\Z^n)_{\vis,\tilde{S}}) = \Lambda(\cW,g\mathbb{Z}^n)_{\vis,S}.
\end{equation*}
For every finite set $S \subset\Lambda(\cW,g\Z^n)$ and $\tilde{S} \df \left(\pi_{\phys}\middle|_{g\Z^n}\right)^{-1}\!(S)$. In particular, for $\mu$-a.e. cut-and-project set $\Lambda(\cW,\cL)$, we have
$$\Lambda(\cW,\cL_{\vis,\tilde{S}})= \Lambda(\cW,\cL)_{\vis,S}.$$
\end{lemma}
\begin{proof}
The proof is analogous to \cite[Proposition~3.3]{gringlaz_error_2025} in combination with \Cref{lemma:condition_for_equality}. 
\end{proof}

\begin{proof}[Proof of \Cref{lemma:condition_for_equality}]
Assume that the inclusion in \eqref{eq:inclusion_lemma} 
is strict. Thus, there exists a point
\[
\vx_1 \in \Lambda\!\left(\cW,\cL_{\vis,\tilde{S}}\right)\!\setminus \Lambda(\cW,\cL)_{\vis,S},
\]
such that $\vx_1\notin S$.
By definition of $\Lambda\!\left(\cW,\cL_{\vis,\tilde{S}}\right)$, there exists $\vy_1\in \cL_{\vis,\tilde{S}}$ such that
$$\pi_{\phys}(\vy_1)=\vx_1, \qquad \pi_{\internal}(\vy_1)\in \cW.$$
Since $\vy_1\in \cL_{\vis,\tilde{S}}$, the point $\vy_1$ is visible from every point $\tilde{\vs}\in \tilde{S}$. On the other hand, $\vx_1 \notin \Lambda(\cW,\cL)_{\vis, S}$, so there exist $\vs_0\in S$ and $t\in (0,1)$ such that
\begin{equation}\label{eq: projection_of_y2}
\vx_2\df t\vx_1 + (1-t)\vs_0\in \Lambda(\cW,\cL).    
\end{equation}
Since $\vx_2 \in \Lambda(\cW,\cL)$, there exists $\vy_2 \in \cL$, such that $\pi_{\phys}(\vy_2)=\vx_2$.
By \eqref{eq: projection_of_y2}, we have
\[
\pi_{\phys}(\vy_2-\tilde{\vs}_0)= t\pi_{\phys}(\vy_1- \tilde{\vs}_0).
\]
Define
\[
\vv_1\df \vy_1-\tilde{\vs}_0,\qquad \vv_2\df \vy_2-\tilde{\vs}_0.
\]
Thus, we have $\vv_1,\vv_2\in \cL$ and $\pi_{\phys}(\vv_2)= t\pi_{\phys}(\vv_1)$. Since $\vx_1 \neq \vs_0$, we have 
\begin{equation}
\label{eq:non_trivial_projection}
\pi_{\phys}(\vv_1)= \vx_1-\vs_0\neq \zero.
\end{equation}
Hence $\vv_1\neq \zero$ and $\vv_2\neq \zero$. The vectors $\vv_1,\vv_2$ are linearly independent. In fact, if $\vv_2= c\vv_1$ for some $c\in \R$,  applying $\pi_{\phys}$ gives
$c=t$ and  $\vv_2= t\vv_1$, implying 
\begin{equation*}
\begin{split}
\vy_2- \tilde{\vs}_0 &= t(\vy_1- \tilde{\vs}_0)\\
\vy_2&= t \vy_1 + (1-t)\tilde{\vs}_0,
\end{split}    
\end{equation*}
which contradicts that $\vy_1\in \cL_{\vis,\tilde{S}}$. 

Therefore, letting
\[
U\df \, {\mathrm{span}}_{\mathbb{R}}(\vv_1,\vv_2) \subset \R^{d+m},
\]
we have that 
\[
\dim U=2.
\]
By \eqref{eq:non_trivial_projection} we obtain that $\dim(\pi_{\phys}(U))=1$, and since the kernel of $\pi_{\phys}$ is the space $\{\zero\} \times \mathbb{R}^m$ then,
\[
\dim(\mathrm{span}_\R(\vv_1,\vv_2)\medcap (\{\zero\}\times\mathbb{R}^m))=\dim {\rm ker}(\pi_{\phys}|_{U})=1.
\]
\end{proof}

\begin{lemma}
\label{lemma:density_cps_latt_vis}
Let $n=d+m$ for positive integers $d \ge 2$, $m \ge 1$, and $\cL\subset \R^n$ be a lattice. Let $\tilde{S} \subset \cL$ be a finite set and 
let $\cW \subset \R^m$ be a bounded convex window of positive Lebesgue measure. Suppose that $\Lambda(\cW,\cL)$ is an irreducible cut-and-project set.
Then the density of 
$\Lambda\!\left(\cW,\cL_{\vis,\tilde{S}}\right)$ 
along $\{TD\}_{T>0}$ exists and it is equal to
\[
\theta\!\left(
\Lambda\!\left(
\cW,\cL_{\vis,\tilde{S}}
\right)
\right) 
= 
\frac{\vol{(\cW)}}{\covol{(\cL)}} \prod_{p \in \cP}\left(1- \frac{s_{\cL}(p)}{p^n} \right),
\]
where $D$ is a bounded convex set with positive measure.
\end{lemma}

\begin{proof}
Consider the groups $G = \R^d$, and $H = \R^m \times \prod_{p \in \cP} H_p$,
where
$$H_p= \quotient{\cL}{p\cL},\qquad p\in \cP.$$
In the group  $H$, each finite factor is equipped with the normalized counting measure.
Both $G$ and $H$ are $\sigma$-compact and LCAGs, and the group $H$ is compactly generated. Let $\cL'$ be the image of the diagonal embedding
$$\cL\hookrightarrow G\times H, \qquad \vy \in \cL \longmapsto (\vy,(\pi_{\cL,p}(\vy))_{p \in \cP}).$$
The restriction $\pi_{\phys}'|_{\cL'}: \cL'\to G$ is injective. Using the fact that $\pi_{\internal}(\cL)$ is dense in $\R^m$ and $H_p\cong (\mathbb{Z}\setminus p\mathbb{Z})^n$, it can be shown that $\pi_{\internal}'(\cL')$ is dense in $H$.
Furthermore, $\cL'$ is a cocompact lattice.
Define 
$$\cW_p \df (\quotient{\cL}{p\cL}) \setminus \pi_{\cL,p} (\tilde{S}).$$
and
$$\cW' \df \cW \times \prod_{p \in \cP} \cW_p.$$
By construction, the associated cut-and-project set $\Lambda(\cW',\cL')$ consists exactly of those $\vx\in \R^d$ such that there exists unique $\vy\in \cL$
$$\pi_{\phys}(\vy)=\vx,\qquad \pi_{\internal}(\vy)\in \cW,\qquad \pi_{\cL,p}(\vy)\notin \pi_{\cL,p}(\tilde{S}),\qquad p\in \cP.$$
Equivalently,
$$\Lambda(\cW,\cL_{\vis,\tilde{S}})= \Lambda(\cW',\cL').$$
If $s_{\cL}(p)=p^n$ for some prime $p$, then $\cW_p= \varnothing$, consequently $\cW'=\varnothing$. Therefore, the \namecref{lemma:density_cps_latt_vis} is immediate. In the rest of the proof, we assume that 
$$s_{\cL}(p)\neq p^n,\qquad p\in \cP.$$
Let $p_i$ be the $i$-th rational prime.
Define
$$\cW_t'= \cW\times \prod_{i=1}^t\cW_{p_i}\times \prod_{i>t}H_{p_i},\qquad t\in \mathbb{N}.$$
For fixed $t\in \mathbb{N}$, the cut-and-project set
$$\Lambda(\cW_{t}',\cL'),$$
is an irreducible cut-and-project set. The density is given by
\begin{equation}
\label{eq:density_of_Lambda_t}
\theta(\Lambda(\cW_{t}',\cL')) = \frac{\vol(\cW_t')}{\covol(\cL')} =  \frac{\vol(\cW)}{\covol(\cL)}\cdot \prod_{i=1}^t\left(1- \frac{s_{\cL}(p_i)}{p_i^n}\right).  
\end{equation}
Since 
\begin{equation}
\label{eq:temp_subsets}
\Lambda(\cW',\cL')
\subseteq 
\Lambda(\cW_t',\cL')
\end{equation}
for every $t\in \mathbb{N}$. By \eqref{eq:density_of_Lambda_t}, we get
$$\overline{\theta}(\Lambda(\cW',\cL'))\leq \frac{\vol(\cW)}{\covol(\cL)}\cdot \prod_{i=1}^t\left(1- \frac{s_{\cL}(p_i)}{p_i^n}\right).$$
Letting $t\to \infty$, we obtain
$$\overline{\theta}(\Lambda(\cW',\cL'))\leq \frac{\vol(\cW)}{\covol(\cL)}\cdot \prod_{i=1}^\infty\left(1- \frac{s_{\cL}(p_i)}{p_i^n}\right).$$
Define
\begin{equation}
\label{eq: At_for_RCPS}
A_t\df \Lambda(\cW_t',\cL')\setminus \Lambda(\cW',\cL').    
\end{equation}
By \cref{eq:density_of_Lambda_t,eq:temp_subsets}, the sets $B=\Lambda(\cW',\cL')$ and $A=\Lambda(\cW'_t,\cL')$ satisfy the hypothesis of \Cref{lemma:aux_identity}. Therefore 
\begin{equation}
\label{eq:last_eq_densities}
\underline{\theta}(\Lambda(\cW',\cL')) = \theta(\Lambda(\cW'_t,\cL')) - \overline{\theta}(A_t).
\end{equation}

We state the following, postpone its proof until  the end, and finish the proof of \Cref{lemma:density_cps_latt_vis}.
\begin{lemma}
\label{lemma:last_step}
Let $A_t$ be as in \eqref{eq: At_for_RCPS}. Then
$\lim_{t\to\infty}\overline{\theta}(A_t)=0$.    
\end{lemma}

Combining the previous \namecref{lemma:last_step} with \eqref{eq:last_eq_densities} and taking the limit as $t\to \infty$,
\[
\overline{\theta}(\Lambda(\cW',\cL')) \le 
\frac{\vol(\cW)}{\covol(\cL)}\cdot \prod_{i=1}^\infty\left(1- \frac{s_{\cL}(p_i)}{p_i^n}\right) 
=
\underline{\theta}(\Lambda(\cW',\cL')).
\]
This proves that the density ${\theta}(\Lambda(\cW,\cL_{\vis,\tilde{S}}))$ along $\{TD\}_{T>0}$ exists and is equal 
\[
\frac{\vol(\cW)}{\covol(\cL)}\cdot \prod_{p\in \cP}\left(1- \frac{s_{\cL}(p)}{p^n}\right)
\]
establishing \Cref{lemma:density_cps_latt_vis}.
\end{proof}

\begin{proof}[Proof of \Cref{lemma:last_step}]
We claim that
\begin{equation}
\label{eq:At_for_random}
A_t\subseteq \bigcup_{\vs\in S}\bigcup_{\substack{p\in \mathcal{P}\\ 
 p>p_t}}(\vs+ \Lambda(\cW-\pi_{\internal}(\tilde{\vs}), p\cL)\setminus \{\vs\}),    
\end{equation}
where ${\tilde\vs} =\left(\pi_{\phys}|_{\cL}\right)^{-1}(\vs)$ is the unique lift of $\vs\in \Lambda$ to $\cL$.
Indeed, let $\vx\in A_t$ be an arbitrary element. Since $\vx\in \Lambda(\cW_t',\cL')$, there is unique $\vy\in \cL$, such that $\vy\neq \tilde{\vs}$ for all $\tilde{\vs}\in \tilde{S}$ and
$$\pi_{\phys}(\vy)= \vx, \quad \pi_{\internal}(\vy)\in \cW,$$
and 
$$\pi_{\cL,p}(\vy)\notin \pi_{\cL,p}(\tilde S),\qquad p\in \mathcal{P},\, p\leq p_t.$$
On the other hand, $\vx\notin \Lambda(\cW',\cL')$ implies that there is a prime $p>p_t$ and $\tilde{\vs}_0\in \tilde{S}$, such that
$$
\vy \equiv \tilde{\vs}_0\, \bmod p\cL. $$
Hence
$$\vx\in \pi_{\phys}(\tilde{\vs}_0) + \Lambda(\cW-\pi_{\internal}(\tilde{\vs}_0), p\cL),\qquad \textup{and}\, \qquad \vx\neq \vs_{0}.$$
This completes the proof of the claim.

For $\vs \in S$ and $p\in \cP$, write
$$
N_{\vs,p,T} \df \card{
\vs +\Lambda(\cW- \pi_{\internal}(\tilde{\vs}),p\cL)\setminus \{\vs\}\medcap TD
}.
$$
Let $\lambda_1(\cL)=\min\{|\vv|\, :\, \vv\in \cL\setminus \{\zero\}\}>0$ be the first successive minima.  Choose constants $R_D, R_S,R_{\cW}>0$ such that
\[
TD\subseteq  B_{R_DT}^d(\zero)
\quad
\text{for all }T>0,
\qquad
|\vs|\leq R_S
\quad 
\text{for all }\vs\in S,
\]
and 
$$
|\vu- \pi_{\internal}(\tilde\vs)|\leq R_{\cW}
\quad
\text{for all }\vs\in S \text{ and all }\vu \in \cW.
$$
Suppose that $N_{\vs,p,T}>0$. Then there exist $\vs \neq  \vx \in TD$, $\vu \in \cW$ 
and 
$\vv\in \cL\setminus\{\zero\}$ 
such that
$$
\vx-\vs= p\pi_{\phys}(\vv),\qquad \vu- \pi_{\internal}(\tilde{\vs})= p\pi_{\internal}(\vv).
$$
Thus,
$$p\vv= (\vx-\vs,\vu-\pi_{\internal}(\tilde{\vs})).$$
It follows that
\begin{equation*}
\begin{split}
p\lambda_1(\cL)&\leq 
|p\vv|=|(\vx-\vs,\vu-\pi_{\internal}(\tilde{\vs}))|\leq |\vx-\vs|+ |\vu- \pi_{\internal}(\tilde{\vs})|\\
&\leq |\vx|+ |\vs|+ R_{\cW}\leq R_DT+ R_S+ R_{\cW}.
\end{split}    
\end{equation*}
For $T>1$, we have $R_DT+ R_S+ R_{\cW}\leq T(R_D+ R_S+ R_{\cW})$. Hence
$$p\leq aT,\qquad a= \frac{R_D+ R_S+ R_{\cW}}{\lambda_1(\cL)}.$$
Since the constant $a$ is independent of $\vs,p,t, T$, there are only finitely many primes $p$ in \eqref{eq:At_for_random} that contribute after intersecting with $TD$.

Every point counted by $N_{\vs,p,T}$ corresponds to a point $\vv\in \cL\setminus \{\zero\}$ satisfying
$$\pi_{\phys}(\vv)\in \frac{TD-\vs}{p},\qquad \pi_{\internal}(\vv)\in \frac{\cW-\pi_{\internal}(\tilde{\vs})}{p}.$$
Note that 
$$\frac{\cW- \pi_{\internal}(\tilde{\vs})}{p}\subseteq B_{R_{\cW}}(\zero),\qquad \vs\in S,\, p\geq 2,$$
and
$$\frac{TD-\vs}{p}\subseteq B_{\frac{(R_D+R_S)T}{p}}(\zero),\qquad \vs\in S,\, p\geq 2,\, T>1.$$
To simplify the notation, we denote
$$B^{T,p}=  B_{\frac{(R_D+R_S)T}{p}}(\zero),\qquad T>1,$$
and 
$$B_{\cW}(\textbf{0})= B_{\cW}.$$
We have
$$N_{\vs,p,T}\leq \card{\Lambda(B_\cW,\cL)\medcap B^{T,p}},\qquad \vs\in S,\, T>1,\,p\in \cP$$
Therefore,
for sufficiently large $T$, by \eqref{eq:At_for_random}, we have
\begin{equation}
\begin{split}
\card{A_t\medcap TD} 
&\leq
\sum_{\vs\in S}
\sum_{p_t<p\leq aT}
N_{\vs,p,T}\\
&\leq \sum_{\vs\in S}
\sum_{p_t<p\leq aT}
\card{\Lambda(B_\cW,\cL)\medcap B^{T,p}}
\end{split}    
\end{equation}
Since $\Lambda(B_\cW,\cL)$ is uniformly discrete, there exists a $\delta>0$ such that
$$|\vx-\vy|>\delta,\qquad \vx\neq \vy\in \Lambda(B_\cW,\cL).$$
Consequently, there exists a constant $C_{\Lambda(B_\cW,\cL)}$ independent of $\vs,p, t, T$ such that
\begin{equation}
\begin{split}
\card{\Lambda(B_\cW,\cL)\medcap B^{T,p}}&\leq C_{\Lambda(B_\cW,\cL)} \left(1+ \frac{2(R_D+ R_S)T}{p}\right)^d\\  
&\leq C_{\Lambda(B_\cW,\cL)}' \left(1+ \frac{T}{p}\right)^d,\\  
\end{split}\qquad T>1,\, p\in \cP
\end{equation}
Using the fact that $p\leq aT$, we get
\begin{equation}
\begin{split}
 \card{\Lambda(B_\cW,\cL)\medcap B^{T,p}}  
&\leq C_{\Lambda(B_\cW,\cL)}''\left(\frac{T}{p}\right)^d,\\  
\end{split}\qquad T>1,\, p\in \cP, p_t<p\leq aT,
\end{equation}
where the constant $C_{\Lambda(B_\cW,\cL)}''$ is independent of $\vs,p,t,T$.
Since $S$ is finite, we get
\begin{equation}
\begin{split}
\card{A_t\medcap TD}
&\leq 
\card{S}C_{\Lambda(B_\cW,\cL)}''
\sum_{p_t<p\leq aT}\left(\frac{T}{p}\right)^d\\
\frac{\card{A_t\medcap TD}}{\vol(TD)}&\leq \card{S}C_{\Lambda(B_\cW,\cL)}'''\sum_{p_t<p\leq aT}\frac{1}{p^d}\\
\frac{\card{A_t\medcap TD}}{\vol(TD)}&\leq \card{S}C_{\Lambda(B_\cW,\cL)}'''\sum_{p_t<p}\frac{1}{p^d},
\end{split}    
\end{equation}
where $C_{\Lambda(B_\cW,\cL)}'''$ is a constant independent of $t,T$.
Taking $\limsup_{T\to \infty}$ gives
\begin{equation}\label{eq: upper_density_of_At_random}
\overline{\theta}(A_t)\leq \card{S}C_{\Lambda(B_\cW,\cL)}'''\sum_{p_t<p}\frac{1}{p^d}.    
\end{equation}
Since $d\geq 2$, 
$$\sum_{p\in \cP}\frac{1}{p^2}<\infty$$
taking $t\to \infty$ in \eqref{eq: upper_density_of_At_random}, we get
$$\lim_{t\to\infty}\overline{\theta}(A_t)=0.$$
\end{proof}

\section{Proof of  \texorpdfstring{
\Cref{prop:equality_affine_case}}{}}

Recall that $\overline{\mu}$ is invariant under translations by vectors in $\R^d$ and assigns full measure to the collection of cut-and-project sets $\Lambda(\cW,\cL)$ for $\cL \in \scrY_n$. The proof follows from the same ideas as the proof of \cite[Proposition 3.1]{gringlaz_error_2025}. For $\overline{\mu}$-a.e. cut-and-project set $\Lambda(\cW,\cL)$ and $\vs\in S$, we have
$$\Lambda(\cW,\cL)_{\vis,\{\vs\}}= \Lambda(\cW,\cL)\setminus \{\vs\}.$$
By the definition of visibility from the set $S$, we have
$$\Lambda(\cW,\cL)_{\vis,S}= \medcap_{\vs\in S}\Lambda(\cW,\cL)_{\vis,\{\vs\}}.$$
Since $S$ is countable,  for $\overline{\mu}$-a.e. cut-and-project set $\Lambda(\cW,\cL)$
$$
\Lambda(\cW,\cL)_{\vis,S}= \medcap_{\vs\in S}\Lambda(\cW,\cL)_{\vis,\{\vs\}}= \medcap_{\vs\in S}\Lambda(\cW,\cL)\setminus \{\vs\}= \Lambda(\cW,\cL)\setminus S.
$$ 

\printbibliography
\end{document}